\documentclass{tran-l}

\usepackage{amssymb}
\usepackage[english]{babel}
\usepackage[T1]{fontenc}
\usepackage{color}
\usepackage{enumerate}
\usepackage{mathtools}
\usepackage{url}
\usepackage[caption=false]{subfig}
\usepackage{pgfplots}
	\pgfplotsset{compat=newest}
\usepackage{tikz}

\newtheorem{theorem}{Theorem}[section]
\newtheorem{lemma}[theorem]{Lemma}
\newtheorem{proposition}[theorem]{Proposition}
\newtheorem{corollary}[theorem]{Corollary}

\theoremstyle{definition}

\newtheorem*{example}{Example}

\theoremstyle{remark}
\newtheorem{remark}[theorem]{Remark}

\numberwithin{equation}{section}

\newcommand{\ip}[2]{\langle#1,#2\rangle}
\newcommand{\ipX}[2]{\langle#1,#2\rangle_{X}}
\newcommand{\ipY}[2]{\langle#1,#2\rangle_{Y}}

	\let\abs=\envert

	\let\norm=\enVert
\newcommand{\normX}[1]{\lVert#1\rVert_{X}}
\newcommand{\normY}[1]{\lVert#1\rVert_{Y}}

\newcommand{\RR}{\mathbb{R}}
\newcommand{\LL}{\mathcal{L}}
\newcommand{\OO}{\mathcal{O}}

\renewcommand\d{\mathop{}\!\mathrm{d}}
\DeclareMathOperator{\e}{\mathrm{e}}
\DeclareMathOperator{\sri}{\mathrm{sri}}

\DeclareMathOperator{\argminmax}{\mathrm{argminmax}}
\DeclareMathOperator{\proj}{\mathrm{proj}}
\DeclareMathOperator{\id}{\mathrm{Id}}
\DeclareMathOperator{\zer}{\mathrm{zer}}

\copyrightinfo{}{}

\begin{document}

\title[Asymptotic behavior of the (AHT) differential system]{%
	Asymptotic behavior of the Arrow--Hurwicz differential system with Tikhonov regularization
}
\author[F.\ Battahi]{Fouad Battahi}
\address{%
	Cadi Ayyad University, Faculty of Sciences Semlalia, 40000 Marrakech,
	Morocco
}
\email{f.battahi.ced@uca.ac.ma}
\author[Z.\ Chbani]{Zaki Chbani}
\address{%
	Cadi Ayyad University, Faculty of Sciences Semlalia, 40000 Marrakech,
	Morocco
}
\email{chbaniz@uca.ac.ma}
\author[S. K.\ Niederl{\"a}nder]{Simon K.\ Niederl{\"a}nder}
\address{%
	Ingolstadt University of Applied Sciences, 85049 Ingolstadt, Germany
}
\email{simon.niederlaender@thi.de}
\author[H. Riahi]{Hassan Riahi}
\address{%
	Cadi Ayyad University, Faculty of Sciences Semlalia, 40000 Marrakech,
	Morocco
}
\email{h-riahi@uca.ac.ma}


\dedicatory{Dedicated to the memory of H{\'e}dy Attouch}

\begin{abstract}
	In a real Hilbert space setting, we investigate the asymptotic behavior of the solutions of the classical Arrow--Hurwicz differential system combined with Tikhonov regularizing terms. Under some newly proposed conditions on the Tikhonov terms involved, we show that the solutions of the regularized Arrow--Hurwicz differential system strongly converge toward the element of least norm within its set of zeros. Moreover, we provide fast asymptotic decay rate estimates for the so-called primal-dual gap function and the norm of the solutions' velocity. If, in addition, the Tikhonov regularizing terms are decreasing, we provide some refined estimates in the sense of an exponentially weighted moving average. Under the additional assumption that the governing operator of the Arrow--Hurwicz differential system satisfies a reverse Lipschitz condition, we further provide a fast rate of strong convergence of the solutions toward the unique zero. We conclude our study by deriving the corresponding decay rate estimates with respect to the so-called viscosity curve. Numerical experiments illustrate our theoretical findings.
	\vspace{-25pt}
\end{abstract}

\maketitle
\section{Introduction}\label{sec1}
Let $X$ and $Y$ be real Hilbert spaces endowed with inner products $\ipX{\,\cdot\,}{\,\cdot\,}$, $\ipY{\,\cdot\,}{\,\cdot\,}$ and associated norms $\normX{\,\cdot\,}$, $\normY{\,\cdot\,}$. Consider the minimization problem
\begin{equation}
	\renewcommand{\theequation}{P}\tag{\theequation}\label{sec1:cvxproblem}
	\min\,\{f(x):Ax=b\},
\end{equation}
where $f:X\to\RR$ is a convex and continuously differentiable function, $A:X\to Y$ a linear and continuous operator, and $b\in Y$. We assume that the (closed and convex) set of optimal solutions of \eqref{sec1:cvxproblem} is non-empty, i.e.,
\begin{equation*}
	S:=\{x\in C:f(x)=\inf\nolimits_{C}f\}\neq\emptyset
\end{equation*}
with $C:=\{x\in X: Ax=b\}$ denoting the feasible set of \eqref{sec1:cvxproblem}. Recall that \eqref{sec1:cvxproblem} admits an optimal solution whenever $C$ is non-empty and, for instance, $f$ is coercive, that is, $\lim_{\normX{x}\to+\infty}f(x)=+\infty$.

Let us associate with \eqref{sec1:cvxproblem} the Lagrangian
\begin{align*}
	L:X\times Y&\longrightarrow\RR\\
	(x,\lambda)&\longmapsto f(x)+\ipY{\lambda}{Ax-b}
\end{align*}
which, by construction, is a convex-concave and continuously differentiable bifunction. Classically, the convex minimization problem \eqref{sec1:cvxproblem} admits an equivalent representation in terms of the saddle-value problem
\begin{equation*}
	\min_{x\in X}\sup_{\lambda\in Y}L(x,\lambda).
\end{equation*}
It is well known (see, e.g., Ekeland and T{\'e}mam \cite{IE-RT:99}) that $\bar{x}\in X$ is an optimal solution of \eqref{sec1:cvxproblem}, and $\bar{\lambda}\in Y$ a corresponding Lagrange multiplier, if and only if $(\bar{x},\bar{\lambda})$ is a saddle point of $L$, that is,
\begin{equation*}
	L(\bar{x},\lambda)\leq L(\bar{x},\bar{\lambda})\leq
	L(x,\bar{\lambda})\quad\forall(x,\lambda)\in X\times Y.
\end{equation*}
Equivalently, $(\bar{x},\bar{\lambda})\in X\times Y$ is a saddle point of $L$ if and only if $(\bar{x},\bar{\lambda})$ satisfies the system of primal-dual optimality conditions
\begin{align*}
	\begin{cases}
		\begin{aligned}
			\nabla f(x)+A^{\ast}\lambda&=0\\[.75ex]
			Ax-b&=0
		\end{aligned}
	\end{cases}
\end{align*}
with $\nabla f:X\to X$ denoting the gradient of $f$, and $A^{\ast}:Y\to X$ the adjoint operator of $A$. Throughout, we denote by $M\subset Y$ the (possibly empty, closed, and convex) set of Lagrange multipliers associated with \eqref{sec1:cvxproblem}. Recall that a Lagrange multiplier (and thus, a saddle point of $L$) exists, for example, whenever the constraint qualification
\begin{equation*}
	b\in\sri A(X)
\end{equation*}
is verified. Here, for a convex set $K\subset Y$, we denote by
\begin{equation*}
	\sri K =\Bigl\{x\in K:\bigcup_{\mu>0}\mu(K-x)\
	\text{is a closed linear subspace of $Y$}\Bigr\}
\end{equation*}
its strong relative interior; we refer the reader to Bauschke and Combettes \cite{HHB-PLC:17} (see also Bo\c{t} \cite{RIB:10}) for a detailed exposition of constraint qualifications.

In this work, we investigate the nonautonomous differential system
\begin{equation}
	\renewcommand{\theequation}{AHT}\tag{\theequation}\label{sec1:arrhurtik}
	\begin{cases}
		\begin{aligned}
			\dot{x}+\nabla f(x)+A^{\ast}\lambda+\varepsilon(t)x&=0\\
			\dot{\lambda}+b-Ax+\varepsilon(t)\lambda&=0
		\end{aligned}
	\end{cases}
\end{equation}
relative to the convex minimization problem \eqref{sec1:cvxproblem}. The \eqref{sec1:arrhurtik} evolution system essentially combines the classical ``generalized steepest descent dynamics'' introduced by Arrow and Hurwicz \cite{KJA-LH:51} (see also Arrow et al.~\cite{KJA-LH-HU:58} and Kose \cite{TK:56}) with Tikhonov regularizing terms; cf. Tikhonov and Ars{\'e}nine \cite{ANT-VYA:74}. Here, $\varepsilon:[t_{0},+\infty[\ \to\ ]0,+\infty[$ denotes, for some $t_{0}\geq0$, the Tikhonov regularization function which is assumed to be continuously differentiable such that
\begin{equation*}
	\lim_{t\to+\infty}\varepsilon(t)=0.
\end{equation*}
In view of this regularization, the \eqref{sec1:arrhurtik} differential system is governed by the perturbed operator
\begin{equation*}
	T_{t}:=T+\varepsilon(t)\id,
\end{equation*}
where
\begin{align*}
	T:X\times Y&\longrightarrow X\times Y\\
	(x,\lambda)&\longmapsto
	\big(\nabla_{x}L(x,\lambda),-\nabla_{\lambda}L(x,\lambda)\big)
\end{align*}
is the maximally monotone operator associated with the ``saddle function'' $L$; see, e.g., Rockafellar \cite{RTR:70,RTR:71}. Noticing that the zeros of $T$ are nothing but the saddle points of $L$, i.e.,
\begin{equation*}
	\zer T=S\times M,
\end{equation*}
and the Tikhonov regularization function $\varepsilon(t)$ is vanishing as $t\to+\infty$, we may expect that the solutions $(x(t),\lambda(t))$ of \eqref{sec1:arrhurtik} converge, as $t\to+\infty$, toward an element in $S\times M$.

As it turns out, the asymptotic behavior of the solutions of \eqref{sec1:arrhurtik} depends critically on the rate at which $\varepsilon(t)$ tends to zero as $t\to+\infty$. In the particular case when $\varepsilon(t)$ vanishes ``sufficiently fast'' as $t\to+\infty$, in the sense that
\begin{equation*}
	\int_{t_{0}}^{\infty}\varepsilon(\tau)\d\tau<+\infty,
\end{equation*}
the solutions of \eqref{sec1:arrhurtik} are known to inherit the asymptotic properties of the ones of the classical Arrow--Hurwicz differential system; see, e.g., the general results in Attouch et al.~\cite{HA-AC-MOC:18} (see also Cominetti et al.~\cite{RC-JP-SS:08}). As such, we may only expect the weak ergodic convergence of the solutions of $\eqref{sec1:arrhurtik}$ toward their asymptotic center in $S\times M$; see Niederl{\"a}nder \cite{SKN:21,SKN:23} for the corresponding results on the classical Arrow--Hurwicz evolution system.

On the other hand, if the Tikhonov regularization function $\varepsilon(t)$ vanishes ``slowly'' as $t\to+\infty$, in the sense that
\begin{equation*}
	\int_{t_{0}}^{\infty}\varepsilon(\tau)\d\tau=+\infty,
\end{equation*}
the solutions of \eqref{sec1:arrhurtik} are asymptotically dominated by the regularizing terms. In this case, the solutions of \eqref{sec1:arrhurtik} are known to be strongly convergent toward the element of least norm in $S\times M$, provided that $\varepsilon(t)$ satisfies, in addition, the ``finite-length property'' (see Cominetti et al.~\cite{RC-JP-SS:08})
\begin{equation*}
	\int_{t_{0}}^{\infty}\abs{\dot{\varepsilon}(\tau)}\d\tau<+\infty,
\end{equation*}
or the ``limiting condition'' (see Bo\c{t} and Nguyen \cite{RIB-DKN:24})
\begin{equation*}
	\lim_{t\to+\infty}\frac{\abs{\dot{\varepsilon}(t)}}{\varepsilon(t)}=0.
\end{equation*}
We note that the strong convergence has been previously established under the assumption that $\varepsilon(t)$ is decreasing with
\begin{equation*}
	\lim_{t\to+\infty}\frac{\dot{\varepsilon}(t)}{\varepsilon^{2}(t)}=0;
\end{equation*}
see, e.g., Israel Jr.~and Reich \cite{MMI-SR:81}, Attouch and Cominetti \cite{HA-RC:96} (see also Browder \cite{FEB:76} and Reich \cite{SR:77}). Ever since, the subject of combining first- and second-order dynamics with Tikhonov regularizing terms has gained significant attention. We only mention here the recent works of Battahi et al.~\cite{FB-ZC-HR:24,FB-ZC-HR:25}, Bo\c{t} and Nguyen \cite{RIB-DKN:24} in the context of first-order differential systems, and Attouch and Czarnecki \cite{HA-MOC:02}, Attouch et al.~\cite{HA-AB-ZC-HR:22,HA-ZC-HR:18} for earlier studies on second-order evolution systems.

In this work, we focus on the derivation of fast convergence rates for the solutions of the \eqref{sec1:arrhurtik} differential system. Under the assumption that $\varepsilon(t)$ is twice continuously differentiable such that there exists $t_{+}\geq t_{0}$ with
\begin{align*}
	\begin{rcases}
		\begin{aligned}
			\varepsilon^{2}(t)+\dot{\varepsilon}(t)&\geq0\\
			2\varepsilon(t)\dot{\varepsilon}(t)+\ddot{\varepsilon}(t)&\leq0
			\hspace{1pt}
		\end{aligned}
	\end{rcases}\quad\forall t\geq t_{+},
\end{align*}
we show that the solutions $(x(t),\lambda(t))$ of \eqref{sec1:arrhurtik} strongly converge, as $t\to+\infty$, to the element of least norm in $S\times M$, i.e.,
\begin{equation*}
	\lim_{t\to+\infty}(x(t),\lambda(t))=\proj_{S\times M}(0,0).
\end{equation*}
Moreover, we prove that the solutions $(x(t),\lambda(t))$ of \eqref{sec1:arrhurtik} obey, for every $(\bar{x},\bar{\lambda})\in S\times M$, the asymptotic estimates
\begin{align*}
	\norm{(\dot{x}(t),\dot{\lambda}(t))+\varepsilon(t)((x(t),\lambda(t))
		-(\bar{x},\bar{\lambda}))}^{2}&=
	\OO\big(\e^{-2\rho(t)}+\,\varepsilon^{2}(t)\big)\ 
	\text{as $t\to+\infty$};\\
	\varepsilon(t)\big(L(x(t),\bar{\lambda})-L(\bar{x},\lambda(t))\big)&=
	\OO\big(\e^{-2\rho(t)}+\,\varepsilon^{2}(t)\big)\
	\text{as $t\to+\infty$};\\
	\norm{T(x(t),\lambda(t))-T(\bar{x},\bar{\lambda})}^{2}&=
	\OO\big(\e^{-2\rho(t)}+\,\varepsilon^{2}(t)\big)\
	\text{as $t\to+\infty$};\\
	\norm{(\dot{x}(t),\dot{\lambda}(t))}^{2}&=
	\OO\big(\e^{-2\rho(t)}+\,\varepsilon^{2}(t)\big)\
	\text{as $t\to+\infty$},
\end{align*}
with the auxiliary function $\rho:[t_{0},+\infty[\ \to\RR$ being defined by
\begin{equation*}
	\rho(t)=\int_{t_{0}}^{t}\varepsilon(\tau)\d\tau.
\end{equation*}
The latter essentially recovers the decay rate estimate
\begin{equation*}
	\norm{(\dot{x}(t),\dot{\lambda}(t))}^{2}
	=\OO\big(\e^{-2\rho(t)}+\,\abs{\dot{\varepsilon}(t)}\big)\
	\text{as $t\to+\infty$}
\end{equation*}
recently obtained by Bo\c{t} and Nguyen \cite{RIB-DKN:24} in the context of general monotone operator flows with Tikhonov regularization. If, in addition, the Tikhonov regularization function $\varepsilon(t)$ is decreasing, we show that the following refined estimate holds:
\begin{equation*}
	\lim_{t\to+\infty}\e^{-\rho(t)}\int_{t_{+}}^{t}\e^{\rho(t)}
	\frac{1}{\varepsilon(\tau)}\norm{(\dot{x}(\tau),\dot{\lambda}(\tau))}^{2}
	\d\tau<+\infty.
\end{equation*}
In the particular case $\varepsilon(t)=1/t$ with $t_{0}>0$, the above estimate reduces to
\begin{equation*}
	\lim_{t\to+\infty}\frac{1}{t}\int_{t_{0}}^{t}\tau^{2}\norm{(\dot{x}(\tau),\dot{\lambda}(\tau))}^{2}
	\d\tau<+\infty,
\end{equation*}
which suggests a fast decay of the quantity $t\norm{(\dot{x}(t),\dot{\lambda}(t))}^{2}$ as $t\to+\infty$ in the sense of an ``exponentially weighted moving average''. Moreover, under the assumption that there exists $\alpha>0$ such that for every $(x,\lambda),(\xi,\eta)\in X\times Y$, it holds that
\begin{equation}
	\renewcommand{\theequation}{L}\tag{\theequation}\label{sec1:errorbound}
	\norm{T(x,\lambda)-T(\xi,\eta)}^{2}\geq
	\alpha\norm{(x,\lambda)-(\xi,\eta)}^{2},
\end{equation}
we infer that the solutions $(x(t),\lambda(t))$ of \eqref{sec1:arrhurtik} obey, for $(\bar{x},\bar{\lambda})\in S\times M$, the decay rate estimate
\begin{equation*}
	\norm{(x(t),\lambda(t))-(\bar{x},\bar{\lambda})}^{2}=
	\OO\big(\e^{-2\rho(t)}+\,\varepsilon^{2}(t)\big)\
	\text{as $t\to+\infty$}.
\end{equation*}

We conclude our work by deriving similar asymptotic estimates for the \eqref{sec1:arrhurtik} solutions with respect to the viscosity curve $(x_{t},\lambda_{t})$ which is governed by the unique zero of the $\varepsilon(t)$-strongly monotone operator $T_{t}$, viz., for every $t\geq t_{0}$,
\begin{equation*}
	T(x_{t},\lambda_{t})+\varepsilon(t)(x_{t},\lambda_{t})=(0,0).
\end{equation*}
In particular, we show that the solutions $(x(t),\lambda(t))$ of \eqref{sec1:arrhurtik} obey the estimates
\begin{align*}
	\norm{(\dot{x}(t),\dot{\lambda}(t))+\varepsilon(t)((x(t),\lambda(t))
		-(x_{t},\lambda_{t}))}^{2}&=
	\OO\big(\e^{-2\rho(t)}+\,\varepsilon^{2}(t)\big)\ \text{as $t\to+\infty$};\\
	\norm{T(x(t),\lambda(t))-T(x_{t},\lambda_{t})}^{2}&=
	\OO\big(\e^{-2\rho(t)}+\,\varepsilon^{2}(t)\big)\ \text{as $t\to+\infty$},
\end{align*}
relative to the viscosity curve $(x_{t},\lambda_{t})$. If, moreover, $T$ verifies condition \eqref{sec1:errorbound}, we have the following asymptotic estimate:
\begin{equation*}
	\norm{(x(t),\lambda(t))-(x_{t},\lambda_{t})}^{2}=
	\OO\big(\e^{-2\rho(t)}+\,\varepsilon^{2}(t)\big)\ \text{as $t\to+\infty$}.
\end{equation*}
Numerical experiments on a simple yet representative example illustrate the above theoretical findings.

\section{Preliminaries on Tikhonov regularization}\label{sec2}
Let $X\times Y$ be endowed with the Hilbertian product structure $\ip{\,\cdot\,}{\,\cdot\,}=\ipX{\,\cdot\,}{\,\cdot\,}+\ipY{\,\cdot\,}{\,\cdot\,}$ and associated norm $\norm{\,\cdot\,}$. Consider now, for each $t\geq t_{0}$, the regularized Lagrangian
\begin{align*}
	L_{t}:X\times Y&\longrightarrow\RR\\[-1ex]
	(x,\lambda)&\longmapsto L(x,\lambda)
	+\frac{\varepsilon(t)}{2}\big(\normX{x}^{2}-\normY{\lambda}^{2}\big)
\end{align*}
relative to the convex minimization problem \eqref{sec1:cvxproblem}. Observing that $L_{t}$ is $\varepsilon(t)$-strongly convex-concave, it follows that $L_{t}$ admits, for each $t\geq t_{0}$, the unique saddle point $(x_{t},\lambda_{t})\in X\times Y$, i.e.,
\begin{equation*}
	L_{t}(x_{t},\lambda)\leq L_{t}(x_{t},\lambda_{t})\leq L_{t}(x,\lambda_{t})\quad\forall(x,\lambda)\in X\times Y.
\end{equation*}
Equivalently, the system of primal-dual optimality conditions reads
\begin{align*}
	\begin{cases}
		\begin{aligned}
			\nabla f(x_{t})+A^{\ast}\lambda_{t}+\varepsilon(t)x_{t}&=0\\[.75ex]
			Ax_{t}-b-\varepsilon(t)\lambda_{t}&=0.
		\end{aligned}
	\end{cases}
\end{align*}
In view of the latter, we immediately observe that, for each $t\geq t_{0}$, the unique zero of the $\varepsilon(t)$-strongly monotone operator
\begin{align*}
	T_{t}:X\times Y&\longrightarrow X\times Y\\
	(x,\lambda)&\longmapsto
	\big(\nabla_{x}L_{t}(x,\lambda),-\nabla_{\lambda}L_{t}(x,\lambda)\big),
\end{align*}
that is the ``generator'' of the \eqref{sec1:arrhurtik} differential system, is precisely the saddle point of $L_{t}$, that is,
\begin{align*}
	T_{t}(x_{t},\lambda_{t})=(0,0)&\quad\iff\\
	(x_{t},\lambda_{t})=\argminmax_{X\times Y}L_{t}&.
\end{align*}

Let us start our discussion with a preliminary result on the asymptotic behavior of the so-called viscosity curve $(x_{t},\lambda_{t})$ as $t\to+\infty$. The result is adapted from Bruck \cite[Lemma 1]{REB:74} (see also Attouch \cite{HA:96}, Attouch and Cominetti \cite{HA-RC:96}, Cominetti et al.~\cite[Lemma 4]{RC-JP-SS:08}).
\begin{lemma}\label{sec2:lm:vcgeo}
	Let $S\times M$ be non-empty and let $(x_{t},\lambda_{t})=\argminmax_{X\times Y}L_{t}$ for each $t\geq t_{0}$. Then the following assertions hold:
	\begin{enumerate}[(i)]
		\item $t\mapsto(x_{t},\lambda_{t})$ is bounded on $[t_{0},+\infty[$ and
			\begin{equation*}
				\norm{(x_{t},\lambda_{t})}\leq\norm{\proj_{S\times M}(0,0)}
				\quad\forall t\geq t_{0};
			\end{equation*}
		\item it holds that
			\begin{equation*}
				\lim_{t\to+\infty}(x_{t},\lambda_{t})=\proj_{S\times M}(0,0).
			\end{equation*}
	\end{enumerate}
\end{lemma}
\begin{proof}
	{\it (i)} For each $t\geq t_{0}$, let $(x_{t},\lambda_{t})=\argminmax_{X\times Y}L_{t}$ and take $(\bar{x},\bar{\lambda})\in S\times M$. Using that $(x_{t},\lambda_{t})$ is the saddle point of $L_{t}$, we have
	\begin{equation}
		\begin{split}\label{sec2:ltineq}
			0&\geq L_{t}(x_{t},\bar{\lambda})-L_{t}(\bar{x},\lambda_{t})\\
			&=L(x_{t},\bar{\lambda})-L(\bar{x},\lambda_{t})+\frac{\varepsilon(t)}{2}
			\big(\norm{(x_{t},\lambda_{t})}^{2}
			-\norm{(\bar{x},\bar{\lambda})}^{2}\big).
		\end{split}
	\end{equation}
	On the other hand, $(\bar{x},\bar{\lambda})$ is a saddle point of $L$ so that
	\begin{equation*}
		L(x_{t},\bar{\lambda})-L(\bar{x},\lambda_{t})\geq0.
	\end{equation*}
	Combining the above inequalities and subsequently dividing by $\varepsilon(t)/2$ yields
	\begin{equation*}
		\norm{(\bar{x},\bar{\lambda})}^{2}\geq\norm{(x_{t},\lambda_{t})}^{2}.
	\end{equation*}
	The above inequality being true for every $(\bar{x},\bar{\lambda})\in S\times M$, we arrive at the desired estimate.

	{\it (ii)} Let $(x,\lambda)\in X\times Y$ and let $(\bar{x},\bar{\lambda})\in X\times Y$ be a weak sequential cluster point of $(x_{t},\lambda_{t})_{t\geq t_{0}}$, that is, there exists a sequence $t_{n}\to+\infty$ such that $(x_{t_{n}},\lambda_{t_{n}})\rightharpoonup(\bar{x},\bar{\lambda})$ weakly in $X\times Y$ as $n\to+\infty$. Substituting $t$ by $t_{n}$ in inequality \eqref{sec2:ltineq} yields
	\begin{align*}
		\frac{\varepsilon(t_{n})}{2}\norm{(x,\lambda)}^{2}&\geq
		L(x_{t_{n}},\lambda)-L(x,\lambda_{t_{n}})
		+\frac{\varepsilon(t_{n})}{2}\norm{(x_{t_{n}},\lambda_{t_{n}})}^{2}\\
		&\geq L(x_{t_{n}},\lambda)-L(x,\lambda_{t_{n}}).
	\end{align*}
	Observing that $\varepsilon(t_{n})\to0$ as $n\to+\infty$, we obtain
	\begin{align*}
		0&\geq\limsup_{n\to+\infty}
		\big(L(x_{t_{n}},\lambda)-L(x,\lambda_{t_{n}})\big)\\
		&\geq\liminf_{n\to+\infty}\,L(x_{t_{n}},\lambda)
		+\liminf_{n\to+\infty}\,(-L(x,\lambda_{t_{n}}))\\
		&\geq L(\bar{x},\lambda)-L(x,\bar{\lambda})
	\end{align*}
	thanks to the weak lower semi-continuity of $L(\,\cdot\,,\lambda)$ and $-L(x,\,\cdot\,)$, as $L(\,\cdot\,,\lambda)$ and $-L(x,\,\cdot\,)$ are both convex and lower semi-continuous. The above inequalities being true for every $(x,\lambda)\in X\times Y$, we conclude that $(\bar{x},\bar{\lambda})$ is a saddle point of $L$, that is, $(\bar{x},\bar{\lambda})\in S\times M$.

	On the other hand, using {\it (i)} and owing to the weak lower semi-continuity of the norm $\norm{\,\cdot\,}$, we obtain
	\begin{equation*}
		\norm{\proj_{S\times M}(0,0)}
		\geq\liminf_{n\to+\infty}\,\norm{(x_{t_{n}},\lambda_{t_{n}})}
		\geq\norm{(\bar{x},\bar{\lambda})},
	\end{equation*}
	implying that $(\bar{x},\bar{\lambda})=\proj_{S\times M}(0,0)$. Consequently, $\proj_{S\times M}(0,0)$ is the only possible weak sequential cluster point of $(x_{t},\lambda_{t})_{t\geq t_{0}}$ so that $(x_{t},\lambda_{t})\rightharpoonup\proj_{S\times M}(0,0)$ weakly in $X\times Y$ as $t\to+\infty$. Upon relying on {\it (i)} again, we have
	\begin{align*}
		\norm{\proj_{S\times M}(0,0)}&
		\geq\limsup_{t\to+\infty}\,\norm{(x_{t},\lambda_{t})}\\
		&\geq\liminf_{t\to+\infty}\,\norm{(x_{t},\lambda_{t})}
		\geq\norm{\proj_{S\times M}(0,0)}
	\end{align*}
	and thus,
	\begin{equation*}
		\lim_{t\to+\infty}\norm{(x_{t},\lambda_{t})}=\norm{\proj_{S\times M}(0,0)}.
	\end{equation*}
	Now, as we have both $(x_{t},\lambda_{t})\rightharpoonup\proj_{S\times M}(0,0)$ weakly in $X\times Y$ and $\norm{(x_{t},\lambda_{t})}\to\norm{\proj_{S\times M}(0,0)}$ as $t\to+\infty$, we classically deduce
	\begin{equation*}
		\lim_{t\to+\infty}(x_{t},\lambda_{t})=\proj_{S\times M}(0,0),
	\end{equation*}
	concluding the result.
\end{proof}

\begin{remark}
	In view of the above result, we readily observe that $(x_{t},\lambda_{t})$ obeys, for every $(\bar{x},\bar{\lambda})\in S\times M$, the asymptotic estimate
	\begin{equation*}
		L(x_{t},\bar{\lambda})-L(\bar{x},\lambda_{t})=\OO(\varepsilon(t))\
		\text{as $t\to+\infty$}.
	\end{equation*}
	We show in Section \ref{sec4} that a comparable estimate holds with respect to the solutions of the \eqref{sec1:arrhurtik} differential system.
\end{remark}

\begin{remark}
	We note that the strong convergence of $(x_{t},\lambda_{t})$ toward $\proj_{S\times M}(0,0)$ as $t\to+\infty$ may also be deduced from the perturbed operator
	\begin{equation*}
		T_{t}=T+\varepsilon(t)\id
	\end{equation*}
	by using the graph-closedness property of the maximally monotone operator $T$ with respect to the weak-strong topology; see, e.g., Br{\'e}zis \cite[Theorem 2.2]{HB:73}, Bauschke and Combettes \cite[Theorem 23.44]{HHB-PLC:17}.
\end{remark}

The following result, adapted from Attouch \cite[Proposition 5.3]{HA:96} (see also Attouch and Cominetti \cite{HA-RC:96}, Torralba \cite[Lemma 5.2]{DT:96}, Attouch et al.~\cite[Lemma 2]{HA-AB-ZC-HR:22}), provides some differential information on the viscosity curve $(x_{t},\lambda_{t})$.
\begin{lemma}\label{sec2:lm:vcdiff}
	Let $(x_{t},\lambda_{t})=\argminmax_{X\times Y}L_{t}$ for each $t\geq t_{0}$. Then $t\mapsto(x_{t},\lambda_{t})$ is Lipschitz continuous on the compact intervals of $[t_{0},+\infty[$ and
	\begin{equation*}
		-\,\dot{\varepsilon}(t)
		\ip{(x_{t},\lambda_{t})}{(\dot{x}_{t},\dot{\lambda}_{t})}\geq
		\varepsilon(t)\norm{(\dot{x}_{t},\dot{\lambda}_{t})}^{2}
		\quad\text{a.e. $t\geq t_{0}$}.
	\end{equation*}
\end{lemma}
\begin{proof}
	Let $(x_{t},\lambda_{t})=\argminmax_{X\times Y}L_{t}$ and $(x_{s},\lambda_{s})=\argminmax_{X\times Y}L_{s}$ for~some $t>s\geq t_{0}$. Utilizing that $L_{t}$ is $\varepsilon(t)$-strongly convex-concave, we have
	\begin{align*}
		0&\geq L_{t}(x_{t},\lambda_{s})-L_{t}(x_{s},\lambda_{t})
		+\frac{\varepsilon(t)}{2}
		\norm{(x_{t},\lambda_{t})-(x_{s},\lambda_{s})}^{2}\\
		&=L(x_{t},\lambda_{s})-L(x_{s},\lambda_{t})+\varepsilon(t)
		\ip{(x_{t},\lambda_{t})}{(x_{t},\lambda_{t})-(x_{s},\lambda_{s})}.
	\end{align*}
	Similarly, $L_{s}$ is $\varepsilon(s)$-strongly convex-concave so that
	\begin{equation*}
		0\geq L(x_{s},\lambda_{t})-L(x_{t},\lambda_{s})+\varepsilon(s)
		\ip{(x_{s},\lambda_{s})}{(x_{s},\lambda_{s})-(x_{t},\lambda_{t})}.
	\end{equation*}
	Combining the above inequalities gives
	\begin{equation*}
		0\geq\ip{\varepsilon(t)(x_{t},\lambda_{t})
		-\varepsilon(s)(x_{s},\lambda_{s})}
		{(x_{t},\lambda_{t})-(x_{s},\lambda_{s})}.
	\end{equation*}
	Equivalently, we have
	\begin{equation}
		\begin{gathered}\label{sec2:diffineq}
			0\geq(\varepsilon(t)-\varepsilon(s))\ip{(x_{t},\lambda_{t})}
			{(x_{t},\lambda_{t})-(x_{s},\lambda_{s})}\\
			+\,\varepsilon(s)\norm{(x_{t},\lambda_{t})-(x_{s},\lambda_{s})}^{2}.
		\end{gathered}
	\end{equation}
	Since $\varepsilon(t)$ is continuously differentiable, it is Lipschitz continuous on the compact intervals of $[t_{0},+\infty[$. In view of the above inequality, it readily follows that $(x_{t},\lambda_{t})$ is Lipschitz continuous on the compact intervals of $[t_{0},+\infty[$ as well and thus, differentiable almost everywhere. Upon dividing inequality \eqref{sec2:diffineq} by $(t-s)^{2}$ and letting $s\to t$, for almost every $t\geq t_{0}$, we obtain
	\begin{equation*}
		0\geq\dot{\varepsilon}(t)\ip{(x_{t},\lambda_{t})}
		{(\dot{x}_{t},\dot{\lambda}_{t})}
		+\varepsilon(t)\norm{(\dot{x}_{t},\dot{\lambda}_{t})}^{2},
	\end{equation*}
	concluding the desired inequality.
\end{proof}

\begin{remark}
	In view of the Cauchy--Schwarz inequality, we readily deduce from the above result that $(\dot{x}_{t},\dot{\lambda}_{t})$ obeys the asymptotic estimate
	\begin{equation*}
		\norm{(\dot{x}_{t},\dot{\lambda}_{t})}=
		\OO\Big(\frac{\abs{\dot{\varepsilon}(t)}}{\varepsilon(t)}\Big)\
		\text{as $t\to+\infty$}.
	\end{equation*}
\end{remark}

Let us next investigate the viscosity curve under the additional assumption that $\varepsilon(t)$ is twice continuously differentiable such that
\begin{equation*}
	\lim_{t\to+\infty}\varepsilon(t)=0.
\end{equation*}
The following result provides a fast decay rate estimate on the quantity $\varepsilon(t)\norm{(\dot{x}_{t},\dot{\lambda}_{t})}^{2}$ in the sense of an ``exponentially weighted moving average''.
\begin{lemma}
	Let $S\times M$ be non-empty and let $(x_{t},\lambda_{t})=\argminmax_{X\times Y}L_{t}$ for each $t\geq t_{0}$. Suppose that there exists $t_{+}\geq t_{0}$ such that
	\begin{align*}
		\begin{rcases}
			\begin{aligned}
				\varepsilon^{2}(t)+\dot{\varepsilon}(t)&\geq0\\
				2\varepsilon(t)\dot{\varepsilon}(t)+\ddot{\varepsilon}(t)&\leq0
				\hspace{1pt}
			\end{aligned}
		\end{rcases}\quad\forall t\geq t_{+}.
	\end{align*}
	Then, as $t\to+\infty$, it holds that
	\begin{equation*}
		\e^{-2\rho(t)}\int_{t_{+}}^{t}\e^{2\rho(\tau)}\varepsilon(\tau)
		\norm{(\dot{x}_{\tau},\dot{\lambda}_{\tau})}^{2}\d\tau=
		\OO\big(\e^{-2\rho(t)}+\,\varepsilon^{2}(t)\big).
	\end{equation*}
\end{lemma}
\begin{proof}
	Let $(x_{t},\lambda_{t})=\argminmax_{X\times Y}L_{t}$ and take $(\bar{x},\bar{\lambda})\in S\times M$. Let $\sigma:[t_{0},+\infty[\ \to\RR$ be defined by $\sigma(t)=\varepsilon^{2}(t)+\dot{\varepsilon}(t)$ such that $\dot{\sigma}(t)=2\varepsilon(t)\dot{\varepsilon}(t)+\ddot{\varepsilon}(t)$. In view of the system of primal-dual optimality conditions, for almost every $t\geq t_{0}$, we have\footnote{%
		In the following, we assume that $t\mapsto T(x_{t},\lambda_{t})$ is Lipschitz continuous on the compact intervals of $[t_{0},+\infty[$, implying that it is differentiable almost everywhere. In Section \ref{sec3}, we provide conditions on $T$ which justify this assumption.
	}
	\begin{equation*}
		\Big\langle\frac{\d}{\d t}T(x_{t},\lambda_{t}),
		(\dot{x}_{t},\dot{\lambda}_{t})\Big\rangle
		+\varepsilon(t)\norm{(\dot{x}_{t},\dot{\lambda}_{t})}^{2}
		+\frac{\dot{\varepsilon}(t)}{2}\frac{\d}{\d t}
		\norm{(x_{t},\lambda_{t})}^{2}=0.
	\end{equation*}
	Since $\varepsilon(t)$ is twice continuously differentiable, we readily obtain
	\begin{gather*}
		\frac{1}{2}\frac{\d}{\d t}
		\big(\dot{\varepsilon}(t)\norm{(x_{t},\lambda_{t})}^{2}\big)
		+\varepsilon(t)\big(\dot{\varepsilon}(t)\norm{(x_{t},\lambda_{t})}^{2}\big)
		+\varepsilon(t)\norm{(\dot{x}_{t},\dot{\lambda}_{t})}^{2}\\
		+\,\Big\langle\frac{\d}{\d t}T(x_{t},\lambda_{t}),
		(\dot{x}_{t},\dot{\lambda}_{t})\Big\rangle
		-\frac{\dot{\sigma}(t)}{2}\norm{(x_{t},\lambda_{t})}^{2}=0.
	\end{gather*}
	Multiplying by $\e^{2\rho(t)}$ and taking into account that $T$ is monotone gives
	\begin{gather*}
		\frac{1}{2}\frac{\d}{\d t}
		\big(\e^{2\rho(t)}\dot{\varepsilon}(t)\norm{(x_{t},\lambda_{t})}^{2}\big)
		+\e^{2\rho(t)}\varepsilon(t)\norm{(\dot{x}_{t},\dot{\lambda}_{t})}^{2}\\
		-\,\e^{2\rho(t)}\frac{\dot{\sigma}(t)}{2}\norm{(x_{t},\lambda_{t})}^{2}
		\leq0.
	\end{gather*}
	Integrating over $[t_{+},t]$ and observing that $\dot{\sigma}(t)\leq0$ for all $t\geq t_{+}$, we find that there exists $K\geq0$ such that
	\begin{equation*}
		\frac{\dot{\varepsilon}(t)}{2}\norm{(x_{t},\lambda_{t})}^{2}
		+\e^{-2\rho(t)}\int_{t_{+}}^{t}\e^{2\rho(\tau)}\varepsilon(\tau)
		\norm{(\dot{x}_{\tau},\dot{\lambda}_{\tau})}^{2}\d\tau\leq K\e^{-2\rho(t)}.
	\end{equation*}

	On the other hand, multiplying inequality \eqref{sec2:ltineq} by $\varepsilon(t)$ and subsequently adding it to the above inequality yields
	\begin{gather*}
		\varepsilon(t)\big(L(x_{t},\bar{\lambda})-L(\bar{x},\lambda_{t})\big)
		+\frac{\sigma(t)}{2}\norm{(x_{t},\lambda_{t})}^{2}
		-\frac{\varepsilon^{2}(t)}{2}\norm{(\bar{x},\bar{\lambda})}^{2}\\
		+\,\e^{-2\rho(t)}\int_{t_{+}}^{t}\e^{2\rho(\tau)}\varepsilon(\tau)
		\norm{(\dot{x}_{\tau},\dot{\lambda}_{\tau})}^{2}\d\tau\leq K\e^{-2\rho(t)}.
	\end{gather*}
	Noticing that $(\bar{x},\bar{\lambda})$ is a saddle point of $L$ and using the fact that $\sigma(t)\geq0$ for all $t\geq t_{+}$, we obtain
	\begin{equation*}
		\e^{-2\rho(t)}\int_{t_{+}}^{t}\e^{2\rho(\tau)}\varepsilon(\tau)
		\norm{(\dot{x}_{\tau},\dot{\lambda}_{\tau})}^{2}\d\tau
		\leq K\e^{-2\rho(t)}+\,\frac{\varepsilon^{2}(t)}{2}
		\norm{(\bar{x},\bar{\lambda})}^{2},
	\end{equation*}
	concluding the desired estimate.
\end{proof}

Let us conclude this section with asymptotic decay rate estimates on the viscosity curve $(x_{t},\lambda_{t})$ and its derivative given the additional assumption that there exists $\alpha>0$ such that for every $(x,\lambda),(\xi,\eta)\in X\times Y$, it holds that
\begin{equation}
	\renewcommand{\theequation}{L}\tag{\theequation}
	\norm{T(x,\lambda)-T(\xi,\eta)}^{2}
	\geq\alpha\norm{(x,\lambda)-(\xi,\eta)}^{2}.
\end{equation}
Condition \eqref{sec1:errorbound} may be interpreted as a particular instance of an error-bound or strong metric subregularity condition (see, e.g., Artacho and Geoffroy \cite{FJAA-MHG:08}, Bolte et al.~\cite{JB-TPN-JP-BWS:17}) which clearly implies that $T$ admits at most one zero.
\begin{lemma}\label{sec2:lm:vccondl}
	Let $S\times M$ be non-empty, let $(x_{t},\lambda_{t})=\argminmax_{X\times Y}L_{t}$ for each $t\geq t_{0}$, and suppose that $T:X\times Y\to X\times Y$ satisfies condition \eqref{sec1:errorbound} with $\alpha>0$. Then, for $(\bar{x},\bar{\lambda})\in S\times M$, it holds that
	\begin{align*}
		\norm{(x_{t},\lambda_{t})}^{2}&=
		\OO\Big(\frac{1}{\alpha+\varepsilon^{2}(t)}\Big)\ \text{as $t\to+\infty$};\\
		\norm{(x_{t},\lambda_{t})-(\bar{x},\bar{\lambda})}^{2}&=
		\OO\Big(\frac{\varepsilon^{2}(t)}{\alpha+\varepsilon^{2}(t)}\Big)\
		\text{as $t\to+\infty$};\\
		\norm{(\dot{x}_{t},\dot{\lambda}_{t})}^{2}&=
		\OO\bigg(\frac{\abs{\dot{\varepsilon}(t)}^{2}}
		{\big(\alpha+\varepsilon^{2}(t)\big)^{2}}\bigg)\ \text{as $t\to+\infty$}.
	\end{align*}
\end{lemma}
\begin{proof}
	For each $t\geq t_{0}$, let $(x_{t},\lambda_{t})=\argminmax_{X\times Y}L_{t}$ and take $(\bar{x},\bar{\lambda})\in S\times M$.

	In view of the system of primal-dual optimality conditions, for every $t\geq t_{0}$, it holds that
	\begin{align*}
		\norm{T(0,0)}^{2}&=\norm{T(x_{t},\lambda_{t})-T(0,0)
			+\varepsilon(t)(x_{t},\lambda_{t})}^{2}\\
		&\geq\norm{T(x_{t},\lambda_{t})-T(0,0)}^{2}
		+\varepsilon^{2}(t)\norm{(x_{t},\lambda_{t})}^{2}\\
		&\geq\big(\alpha+\varepsilon^{2}(t)\big)\norm{(x_{t},\lambda_{t})}^{2},
	\end{align*}
	where the first inequality follows from the monotonicity of $T$ while the second one follows from condition \eqref{sec1:errorbound}, concluding the first estimate.

	Again, by virtue of condition \eqref{sec1:errorbound} and the system of primal-dual optimality conditions, for every $t\geq t_{0}$, we have
	\begin{align*}
		\varepsilon^{2}(t)\norm{(\bar{x},\bar{\lambda})}^{2}&=
		\norm{T(x_{t},\lambda_{t})-T(\bar{x},\bar{\lambda})
			+\varepsilon(t)((x_{t},\lambda_{t})-(\bar{x},\bar{\lambda}))}^{2}\\
		&\geq\big(\alpha+\varepsilon^{2}(t)\big)
		\norm{(x_{t},\lambda_{t})-(\bar{x},\bar{\lambda})}^{2}.
	\end{align*}
	Successively dividing by $\varepsilon^{2}(t)$ and passing to the upper limit as $t\to+\infty$ gives the second estimate.

	Consider now $(x_{t},\lambda_{t})=\argminmax_{X\times Y}L_{t}$ and $(x_{s},\lambda_{s})=\argminmax_{X\times Y}L_{s}$ for some $t>s\geq t_{0}$. Utilizing once again condition \eqref{sec1:errorbound}, we have
	\begin{equation*}
		\norm{T(x_{t},\lambda_{t})-T(x_{s},\lambda_{s})}^{2}\geq
		\alpha\norm{(x_{t},\lambda_{t})-(x_{s},\lambda_{s})}^{2}.
	\end{equation*}
	Upon dividing by $(t-s)^{2}$ and letting $s\to t$, for almost every $t\geq t_{0}$, we obtain
	\begin{equation*}
		\Big\lVert\frac{\d}{\d t}T(x_{t},\lambda_{t})\Big\rVert^{2}\geq
		\alpha\norm{(\dot{x}_{t},\dot{\lambda}_{t})}^{2}.
	\end{equation*}
	On the other hand, differentiating the system of primal-dual optimality conditions yields, for almost every $t\geq t_{0}$,
	\begin{equation*}
		\frac{\d}{\d t}T(x_{t},\lambda_{t})
		+\varepsilon(t)(\dot{x}_{t},\dot{\lambda}_{t})
		+\dot{\varepsilon}(t)(x_{t},\lambda_{t})=0.
	\end{equation*}
	Consequently, we have
	\begin{align*}
		\abs{\dot{\varepsilon}(t)}^{2}\norm{(x_{t},\lambda_{t})}^{2}&=
		\Big\lVert\frac{\d}{\d t}T(x_{t},\lambda_{t})
		+\varepsilon(t)(\dot{x}_{t},\dot{\lambda}_{t})\Big\rVert^{2}\\
		&\geq\Big\lVert\frac{\d}{\d t}T(x_{t},\lambda_{t})\Big\rVert^{2}
		+\varepsilon^{2}(t)\norm{(\dot{x}_{t},\dot{\lambda}_{t})}^{2}\\
		&\geq\big(\alpha+\varepsilon^{2}(t)\big)
		\norm{(\dot{x}_{t},\dot{\lambda}_{t})}^{2},
	\end{align*}
	where we again utilized the monotonicity of $T$. Combining the above estimates entails
	\begin{equation*}
		\frac{\abs{\dot{\varepsilon}(t)}^{2}}{\alpha+\varepsilon^{2}(t)}
		\norm{T(0,0)}^{2}\geq\big(\alpha+\varepsilon^{2}(t)\big)
		\norm{(\dot{x}_{t},\dot{\lambda}_{t})}^{2},
	\end{equation*}
	concluding the result.
\end{proof}

\begin{remark}
	We note that similar estimates can be derived under the more general assumption that the perturbed operator $T_{t}=T+\varepsilon(t)\id$ is such that there exists $\alpha:[t_{0},+\infty[\ \to \ ]0,+\infty[$ verifying, for every $(x,\lambda),(\xi,\eta)\in X\times Y$ and $t\geq t_{0}$,
	\begin{equation*}
		\norm{T_{t}(x,\lambda)-T_{t}(\xi,\eta)}^{2}\geq\alpha(t)\norm{(x,\lambda)-(\xi,\eta)}^{2}.
	\end{equation*}
	We leave the details to the reader.
\end{remark}

\section{The \eqref{sec1:arrhurtik} differential system}\label{sec3}
In the following, we presuppose that
\begin{enumerate}[\hspace{6pt}({A}1)]
	\item $f:X\to\RR$ is convex and continuously differentiable;
	\item $\nabla f:X\to X$ is Lipschitz continuous on the bounded subsets of $X$;
	\item $A:X\to Y$ is linear and continuous, and $b\in Y$;
	\item $\varepsilon:[t_{0},+\infty[\ \to \ ]0,+\infty[$ is continuously
		differentiable such that
		\begin{equation*}
			\lim_{t\to+\infty}\varepsilon(t)=0.
		\end{equation*}
\end{enumerate}

Consider again the nonautonomous differential system\footnote{%
	In view of the above assumptions, we readily observe that the governing operator $T:X\times Y\to X\times Y$ of the \eqref{sec1:arrhurtik} differential system is Lipschitz continuous on the bounded subsets of $X\times Y$.
}
\begin{equation}
	\renewcommand{\theequation}{AHT}\tag{\theequation}
	\begin{cases}
		\begin{aligned}
			\dot{x}+\nabla f(x)+A^{\ast}\lambda+\varepsilon(t)x&=0\\
			\dot{\lambda}+b-Ax+\varepsilon(t)\lambda&=0
		\end{aligned}
	\end{cases}
\end{equation}
with initial data $(x_{0},\lambda_{0})\in X\times Y$. Throughout, we assume that \eqref{sec1:arrhurtik} admits, for each $(x_{0},\lambda_{0})\in X\times Y$, a unique (classical) solution, that is, a continuously differentiable function $(x,\lambda):[t_{0},+\infty[\ \to X\times Y$ which verifies \eqref{sec1:arrhurtik} on $[t_{0},+\infty[$ with $(x(t_{0}),\lambda(t_{0}))=(x_{0},\lambda_{0})$ for some $t_{0}\geq0$; see, e.g., Haraux \cite{AH:91}. We refer the reader to Crandall and Pazy \cite{MGC-AP:72}, Furuya et al.~\cite{HF-KM-NK:86}, and Kenmochi \cite{NK:81} for the respective results on nonautonomous evolution equations governed by maximally monotone operators.

Consider again, for each $t\geq t_{0}$, the regularized Lagrangian
\begin{align*}
	L_{t}:X\times Y&\longrightarrow\RR\\[-1ex]
	(x,\lambda)&\longmapsto L(x,\lambda)
	+\frac{\varepsilon(t)}{2}\big(\normX{x}^{2}-\normY{\lambda}^{2}\big)
\end{align*}
associated with the convex minimization problem \eqref{sec1:cvxproblem}. In view of the $\varepsilon(t)$-strong convexity-concavity of the saddle function $L_{t}$, we immediately obtain that for every $(x,\lambda),(\xi,\eta)\in X\times Y$ and $t\geq t_{0}$, it holds that
\begin{equation}
	\begin{gathered}\label{sec3:ttltineq}
		\ip{T_{t}(x,\lambda)}{(x,\lambda)-(\xi,\eta)}
		\geq L_{t}(x,\eta)-L_{t}(\xi,\lambda)\\
		+\,\frac{\varepsilon(t)}{2}\norm{(x,\lambda)-(\xi,\eta)}^{2}.
	\end{gathered}
\end{equation}
Utilizing the above inequality relative to the \eqref{sec1:arrhurtik} evolution system gives the following preliminary estimates with $\rho:[t_{0},+\infty[\ \to\RR$ being defined by
\begin{equation*}
	\rho(t)=\int_{t_{0}}^{t}\varepsilon(\tau)\d\tau.
\end{equation*}
\begin{proposition}\label{sec3:pr:boundedness}
	Let $S\times M$ be non-empty and let $(x,\lambda):[t_{0},+\infty[\ \to X\times Y$ be a solution of \eqref{sec1:arrhurtik}. Then $t\mapsto(x(t),\lambda(t))$ is bounded on $[t_{0},+\infty[$. Moreover, for every $(\bar{x},\bar{\lambda})\in S\times M$, it holds that
	\begin{align*}
		\lim_{t\to+\infty}\e^{-\rho(t)}\int_{t_{0}}^{t}\e^{\rho(\tau)}
		\big(L(x(\tau),\bar{\lambda})-L(\bar{x},\lambda(\tau))\big)\d\tau&<+\infty;
		\\
		\lim_{t\to+\infty}\e^{-\rho(t)}\int_{t_{0}}^{t}\e^{\rho(\tau)}
		\frac{\varepsilon(\tau)}{2}\norm{(x(\tau),\lambda(\tau))}^{2}\d\tau
		&<+\infty.
	\end{align*}
\end{proposition}
\begin{proof}
	Let $(\bar{x},\bar{\lambda})\in S\times M$ and define $\phi:[t_{0},+\infty[\ \to\RR$ by $\phi(t)=\norm{(x(t),\lambda(t))-(\bar{x},\bar{\lambda})}^{2}/2$. Taking the inner product with $(x(t),\lambda(t))-(\bar{x},\bar{\lambda})$ in \eqref{sec1:arrhurtik} and subsequently applying the chain rule yields, for every $t\geq t_{0}$,
	\begin{equation*}
		\dot{\phi}(t)
		+\ip{T_{t}(x(t),\lambda(t))}{(x(t),\lambda(t))-(\bar{x},\bar{\lambda})}=0.
	\end{equation*}
	In view of inequality \eqref{sec3:ttltineq}, we obtain
	\begin{equation*}
		\dot{\phi}(t)+\varepsilon(t)\phi(t)
		+L_{t}(x(t),\bar{\lambda})-L_{t}(\bar{x},\lambda(t))\leq0.
	\end{equation*}
	Equivalently, we have
	\begin{equation}
		\begin{gathered}\label{sec3:bdconvineq}
			\dot{\phi}(t)+\varepsilon(t)\phi(t)
			+L(x(t),\bar{\lambda})-L(\bar{x},\lambda(t))\\
			+\,\frac{\varepsilon(t)}{2}\norm{(x(t),\lambda(t))}^{2}
			\leq\frac{\varepsilon(t)}{2}\norm{(\bar{x},\bar{\lambda})}^{2}.
		\end{gathered}
	\end{equation}
	Multiplying the above inequality by $\e^{\rho(t)}$ yields
	\begin{gather*}
		\frac{\d}{\d t}\big(\e^{\rho(t)}\phi(t)\big)
		+\e^{\rho(t)}\big(L(x(t),\bar{\lambda})-L(\bar{x},\lambda(t))\big)\\
		+\,\e^{\rho(t)}\frac{\varepsilon(t)}{2}\norm{(x(t),\lambda(t))}^{2}
		\leq\frac{1}{2}\frac{\d}{\d t}\e^{\rho(t)}
		\norm{(\bar{x},\bar{\lambda})}^{2}.
	\end{gather*}
	Integrating over $[t_{0},t]$ and subsequently dividing by $\e^{\rho(t)}$ entails
	\begin{equation}
		\begin{gathered}\label{sec3:bdineq}
			\phi(t)+\e^{-\rho(t)}\int_{t_{0}}^{t}\e^{\rho(\tau)}
			\big(L(x(\tau),\bar{\lambda})-L(\bar{x},\lambda(\tau))\big)\d\tau\\
			+\,\e^{-\rho(t)}\int_{t_{0}}^{t}\e^{\rho(\tau)}\frac{\varepsilon(\tau)}{2}
			\norm{(x(\tau),\lambda(\tau))}^{2}\d\tau\leq\e^{-\rho(t)}\phi(t_{0})
			+\frac{1}{2}\norm{(\bar{x},\bar{\lambda})}^{2}.
		\end{gathered}
	\end{equation}
	Noticing that $(\bar{x},\bar{\lambda})$ is a saddle point of $L$, it holds that
	\begin{equation*}
		\e^{-\rho(t)}\int_{t_{0}}^{t}\e^{\rho(\tau)}
		\big(L(x(\tau),\bar{\lambda})-L(\bar{x},\lambda(\tau))\big)\d\tau\geq0.
	\end{equation*}
	Consequently, we have
	\begin{equation*}
		\phi(t)\leq\e^{-\rho(t)}\phi(t_{0})
		+\frac{1}{2}\norm{(\bar{x},\bar{\lambda})}^{2},
	\end{equation*}
	implying that $(x(t),\lambda(t))$ remains bounded on $[t_{0},+\infty[$.

	On the other hand, utilizing inequality \eqref{sec3:bdineq} and taking into account that $\phi(t)\geq0$, we obtain
	\begin{gather*}
		\e^{-\rho(t)}\int_{t_{0}}^{t}\e^{\rho(\tau)}
		\big(L(x(\tau),\bar{\lambda})-L(\bar{x},\lambda(\tau))\big)\d\tau\\
		+\,\e^{-\rho(t)}\int_{t_{0}}^{t}\e^{\rho(\tau)}\frac{\varepsilon(\tau)}{2}
		\norm{(x(\tau),\lambda(\tau))}^{2}\d\tau
		\leq\e^{-\rho(t)}\phi(t_{0})+\frac{1}{2}\norm{(\bar{x},\bar{\lambda})}^{2}.
	\end{gather*}
	Passing to the limit as $t\to+\infty$ entails
	\begin{align*}
		\lim_{t\to+\infty}\e^{-\rho(t)}\int_{0}^{t}\e^{\rho(\tau)}
		\big(L(x(\tau),\bar{\lambda})-L(\bar{x},\lambda(\tau))\big)\d\tau&<+\infty,
		\ \text{and}\\
		\lim_{t\to+\infty}\e^{-\rho(t)}\int_{t_{0}}^{t}\e^{\rho(\tau)}
		\frac{\varepsilon(\tau)}{2}\norm{(x(\tau),\lambda(\tau))}^{2}\d\tau
		&<+\infty,
	\end{align*}
	concluding the desired estimates.
\end{proof}

In view of the above result, the following estimate as outlined in Cominetti et al.~\cite{RC-JP-SS:08} is verified whenever the Tikhonov regularization function $\varepsilon:[t_{0},+\infty[\ \to\ ]0,+\infty[$ is such that
\begin{equation*}
	\int_{t_{0}}^{\infty}\varepsilon(\tau)\d\tau=+\infty.
\end{equation*}
\begin{corollary}\label{sec3:co:limsup}
	Under the hypotheses of Proposition \ref{sec3:pr:boundedness}, suppose that $\varepsilon\notin\LL^{1}([t_{0},+\infty[)$. Then, for every $(\bar{x},\bar{\lambda})\in S\times M$, it holds that
	\begin{equation*}
		\limsup_{t\to+\infty}\,\norm{(x(t),\lambda(t))-(\bar{x},\bar{\lambda})}^{2}
		\leq\norm{(\bar{x},\bar{\lambda})}^{2}
		-\liminf_{t\to+\infty}\,\norm{(x(t),\lambda(t))}^{2}.
	\end{equation*}
\end{corollary}
\begin{proof}
	Recall from inequality \eqref{sec3:bdconvineq} that for every $t\geq t_{0}$, we have
	\begin{gather*}
		\dot{\phi}(t)+\varepsilon(t)\phi(t)
		+L(x(t),\bar{\lambda})-L(\bar{x},\lambda(t))\\
		\leq\frac{\varepsilon(t)}{2}
		\big(\norm{(\bar{x},\bar{\lambda})}^{2}-\norm{(x(t),\lambda(t))}^{2}\big).
	\end{gather*}
	Since $(\bar{x},\bar{\lambda})$ is a saddle point of $L$, it holds that
	\begin{equation*}
		\dot{\phi}(t)+\varepsilon(t)\phi(t)\leq\frac{\varepsilon(t)}{2}
		\big(\norm{(\bar{x},\bar{\lambda})}^{2}-\norm{(x(t),\lambda(t))}^{2}\big).
	\end{equation*}
	Using that $(x(t),\lambda(t))$ remains bounded on $[t_{0},+\infty[$ together with the fact that $\varepsilon\notin\LL^{1}([t_{0},+\infty[)$, applying Lemma \ref{app:lm:lsineq} entails
	\begin{equation*}
		\limsup_{t\to+\infty}\,\phi(t)\leq
		\limsup_{t\to+\infty}\,\frac{1}{2}\big(\norm{(\bar{x},\bar{\lambda})}^{2}
			-\norm{(x(t),\lambda(t))}^{2}\big),
	\end{equation*}
	which is the desired estimate.
\end{proof}

\begin{remark}\label{sec3:rm:liminf}
	Anchoring the above inequality to $\proj_{S\times M}(0,0)$ suggests that the solutions $(x(t),\lambda(t))$ of \eqref{sec1:arrhurtik} strongly converge, as $t\to+\infty$, toward $\proj_{S\times M}(0,0)$ as soon as
	\begin{equation*}
		\liminf_{t\to+\infty}\,\norm{(x(t),\lambda(t))}
		\geq\norm{\proj_{S\times M}(0,0)}.
	\end{equation*}
\end{remark}

Our next result provides sufficient conditions for this inequality to hold assuming that $\varepsilon:[t_{0},+\infty[\ \to\ ]0,+\infty[$ satisfies either one of the following estimates:
\begin{align*}
	\int_{t_{0}}^{\infty}\abs{\dot{\varepsilon}(\tau)}\d\tau&<+\infty,\ \text{or}\\
	\int_{t_{0}}^{\infty}\frac{\abs{\dot{\varepsilon}(\tau)}^{2}}{\varepsilon(\tau)}\d\tau&<+\infty.
\end{align*}
\begin{theorem}\label{sec3:th:zerogap}
	Let $S\times M$ be non-empty, let $(x,\lambda):[t_{0},+\infty[\ \to X\times Y$ be a solution of \eqref{sec1:arrhurtik}, and suppose that $\varepsilon\notin\LL^{1}([t_{0},+\infty[)$ with either $\dot{\varepsilon}\in\LL^{1}([t_{0},+\infty[)$ or $\abs{\dot{\varepsilon}}^{2}/\varepsilon\in\LL^{1}([t_{0},+\infty[)$. Then, for every $(\bar{x},\bar{\lambda})\in S\times M$, it holds that
	\begin{align*}
		\lim_{t\to+\infty}\big(L(x(t),\bar{\lambda})-L(\bar{x},\lambda(t))\big)
		&=0;\\
		\lim_{t\to+\infty}\norm{(\dot{x}(t),\dot{\lambda}(t))}&=0.
	\end{align*}
\end{theorem}
\begin{proof}
	Let $\vartheta:[t_{0},+\infty[\ \to\RR$ be defined by $\vartheta(t)=\norm{(\dot{x}(t),\dot{\lambda}(t))}^{2}/2$. Differentiating $\vartheta(t)$ and taking \eqref{sec1:arrhurtik} into account yields, for almost every $t\geq t_{0}$,
	\begin{equation*}
		\dot{\vartheta}(t)
		+\Big\langle\frac{\d}{\d t}T_{t}(x(t),\lambda(t)),
		(\dot{x}(t),\dot{\lambda}(t))\Big\rangle=0.
	\end{equation*}
	Equivalently, we have
	\begin{equation}
		\begin{gathered}\label{sec3:csinequ}
			\dot{\vartheta}(t)+2\varepsilon(t)\vartheta(t)
			+\Big\langle\frac{\d}{\d t}T(x(t),\lambda(t)),
			(\dot{x}(t),\dot{\lambda}(t))\Big\rangle\\
			+\,\frac{\dot{\varepsilon}(t)}{2}\frac{\d}{\d t}
			\norm{(x(t),\lambda(t))}^{2}=0.
		\end{gathered}
	\end{equation}

	Let us first consider the case when $\dot{\varepsilon}\in\LL^{1}([t_{0},+\infty[)$. Multiplying the above inequality by $\e^{2\rho(t)}$ and using the fact that $T$ is monotone entails
	\begin{equation*}
		\frac{\d}{\d t}\big(\e^{2\rho(t)}\vartheta(t)\big)
		+\e^{2\rho(t)}\frac{\dot{\varepsilon}(t)}{2}\frac{\d}{\d t}
		\norm{(x(t),\lambda(t))}^{2}\leq0.
	\end{equation*}
	In view of the Cauchy--Schwarz inequality, we obtain
	\begin{equation*}
		\frac{\d}{\d t}\big(\e^{2\rho(t)}\vartheta(t)\big)\leq
		\sqrt{2}\e^{\rho(t)}\abs{\dot{\varepsilon}(t)}\norm{(x(t),\lambda(t))}
		\sqrt{\e^{2\rho(t)}\vartheta(t)}.
	\end{equation*}
	Integrating over $[t_{+},t]$ for some fixed $t_{+}\geq t_{0}$ and noting that $(x(t),\lambda(t))$ remains bounded on $[t_{0},+\infty[$, we find that there exists $K\geq0$ such that
	\begin{equation*}
		\e^{2\rho(t)}\vartheta(t)\leq\e^{2\rho(t_{+})}\vartheta(t_{+})
		+\sqrt{2}K\int_{t_{+}}^{t}\e^{\rho(\tau)}\abs{\dot{\varepsilon}(\tau)}
		\sqrt{\e^{2\rho(\tau)}\vartheta(\tau)}\d\tau.
	\end{equation*}
	Successively applying Lemma \ref{app:lm:gtineq} and dividing by $\e^{\rho(t)}$ yields
	\begin{align*}
		\sqrt{\vartheta(t)}&\leq\e^{-(\rho(t)-\rho(t_{+}))}\sqrt{\vartheta(t_{+})}
		+\frac{K}{\sqrt{2}}\e^{-\rho(t)}\int_{t_{+}}^{t}\e^{\rho(\tau)}
		\abs{\dot{\varepsilon}(\tau)}\d\tau\\
		&\leq\e^{-(\rho(t)-\rho(t_{+}))}\sqrt{\vartheta(t_{+})}
		+\frac{K}{\sqrt{2}}\int_{t_{+}}^{t}\abs{\dot{\varepsilon}(\tau)}\d\tau.
	\end{align*}
	Now, since we have both $\varepsilon\notin\LL^{1}([t_{0},+\infty[)$ and $\dot{\varepsilon}\in\LL^{1}([t_{0},+\infty[)$, passing to the upper limit as $t\to+\infty$ entails
	\begin{equation*}
		\limsup_{t\to+\infty}\sqrt{\vartheta(t)}\leq\frac{K}{\sqrt{2}}
		\int_{t_{+}}^{\infty}\abs{\dot{\varepsilon}(\tau)}\d\tau.
	\end{equation*}
	This inequality being true for every $t_{+}\geq t_{0}$, letting $t_{+}\to+\infty$ ensures that $\vartheta(t)$ tends to zero as $t\to+\infty$.

	Let us now consider the case when $\abs{\dot{\varepsilon}}^{2}/\varepsilon\in\LL^{1}([t_{0},+\infty[)$. Multiplying equality \eqref{sec3:csinequ} by $\e^{\rho(t)}$ and using again the fact that $T$ is monotone gives
	\begin{equation*}
		\frac{\d}{\d t}\big(\e^{\rho(t)}\vartheta(t)\big)
		+\e^{\rho(t)}\varepsilon(t)\vartheta(t)
		+\e^{\rho(t)}\frac{\dot{\varepsilon}(t)}{2}\frac{\d}{\d t}
		\norm{(x(t),\lambda(t))}^{2}\leq0.
	\end{equation*}
	Upon applying the Cauchy--Schwarz inequality, we infer
	\begin{equation*}
		\frac{\d}{\d t}\big(\e^{\rho(t)}\vartheta(t)\big)
		\leq\e^{\rho(t)}\frac{\abs{\dot{\varepsilon}(t)}^{2}}{2\varepsilon(t)}
		\norm{(x(t),\lambda(t))}^{2}.
	\end{equation*}
	Integrating over $[t_{+},t]$ for some fixed $t_{+}\geq t_{0}$ and using again that $(x(t),\lambda(t))$ remains bounded on $[t_{0},+\infty[$, we find that there exists $K\geq0$ such that
	\begin{align*}
		\vartheta(t)&\leq\e^{-(\rho(t)-\rho(t_{+}))}\vartheta(t_{+})
		+\frac{K}{2}\e^{-\rho(t)}\int_{t_{+}}^{t}\e^{\rho(\tau)}
		\frac{\abs{\dot{\varepsilon}(\tau)}^{2}}{\varepsilon(\tau)}\d\tau\\
		&\leq\e^{-(\rho(t)-\rho(t_{+}))}\vartheta(t_{+})+\frac{K}{2}\int_{t_{+}}^{t}
		\frac{\abs{\dot{\varepsilon}(\tau)}^{2}}{\varepsilon(\tau)}\d\tau.
	\end{align*}
	Observing now that $\varepsilon\notin\LL^{1}([t_{0},+\infty[)$ and $\abs{\dot{\varepsilon}}^{2}/\varepsilon\in\LL^{1}([t_{0},+\infty[)$, we conclude that $\vartheta(t)$ vanishes as $t\to+\infty$.

	Finally, in view of inequality \eqref{sec3:ttltineq} and the regularized Lagrangian $L_{t}$, for every $(\bar{x},\bar{\lambda})\in S\times M$ and $t\geq t_{0}$, we have
	\begin{gather*}
		\ip{T_{t}(x(t),\lambda(t))}{(x(t),\lambda(t))-(\bar{x},\bar{\lambda})}
		+\frac{\varepsilon(t)}{2}\norm{(\bar{x},\bar{\lambda})}^{2}\\
		\geq L(x(t),\bar{\lambda})-L(\bar{x},\lambda(t)).
	\end{gather*}
	Using again that $(x(t),\lambda(t))$ remains bounded on $[t_{0},+\infty[$, and owing to the fact that $T_{t}(x(t),\lambda(t))\to(0,0)$ strongly in $X\times Y$ as $t\to+\infty$ under each of the above conditions on $\varepsilon(t)$, passing to the limit entails 
	\begin{equation*}
		\lim_{t\to+\infty}\big(L(x(t),\bar{\lambda})-L(\bar{x},\lambda(t))\big)=0,
	\end{equation*}
	concluding the result.
\end{proof}

\begin{remark}
	Let us compare the estimates derived in the proof of Theorem \ref{sec3:th:zerogap} in the ``limiting case'' when $\varepsilon(t)=1/t$ with $t_{0}>0$. On the one hand, for every $t\geq t_{0}$, we have
	\begin{align*}
		\sqrt{\vartheta(t)}&\leq\e^{-\rho(t)}\sqrt{\vartheta(t_{0})}
		+\frac{K}{\sqrt{2}}\e^{-\rho(t)}\int_{t_{0}}^{t}\e^{\rho(\tau)}
		\abs{\dot{\varepsilon}(\tau)}\d\tau\\
		&=\sqrt{\vartheta(t_{0})}\frac{t_{0}}{t}
		+\frac{K}{\sqrt{2}t}\ln\Big(\frac{t}{t_{0}}\Big).
	\end{align*}
	Consequently, $\vartheta(t)$ obeys the asymptotic estimate
	\begin{equation*}
		\vartheta(t)=\OO\Big(\frac{\ln(t)^{2}}{t^{2}}\Big)\ \text{as $t\to+\infty$}.
	\end{equation*}

	On the other hand, for every $t\geq t_{0}$, we have
	\begin{align*}
		\vartheta(t)&\leq\e^{-\rho(t)}\vartheta(t_{0})
		+\frac{K}{2}\e^{-\rho(t)}\int_{t_{0}}^{t}\e^{\rho(\tau)}
		\frac{\abs{\dot{\varepsilon}(\tau)}^{2}}{\varepsilon(\tau)}\d\tau\\
		&=\vartheta(t_{0})\frac{t_{0}}{t}
		+\frac{K}{2t^{2}}\Big(\frac{t}{t_{0}}-1\Big).
	\end{align*}
	In this case, we obtain the comparable decay rate estimate
	\begin{equation*}
		\vartheta(t)=\OO\Big(\frac{1}{t}\Big)\ \text{as $t\to+\infty$}.
	\end{equation*}
\end{remark}

\begin{remark}
	We note that the above result under the condition $\dot{\varepsilon}\in\LL^{1}([t_{0},+\infty[)$ has already been established by Cominetti et al.~\cite[Theorem 9]{RC-JP-SS:08}) using a similar line of arguments. In the recent work of Bo\c{t} and Nguyen \cite{RIB-DKN:24} it has been shown that $(\dot{x}(t),\dot{\lambda}(t))$ also tends to zero as $t\to+\infty$ whenever
	\begin{equation*}
		\lim_{t\to+\infty}\frac{\abs{\dot{\varepsilon}(t)}}{\varepsilon(t)}=0.
	\end{equation*}
\end{remark}

We are now in the position to assert the strong convergence of the solutions of the \eqref{sec1:arrhurtik} differential system.
\begin{proposition}\label{sec3:pr:strongconv}
	Let $S\times M$ be non-empty, let $(x,\lambda):[t_{0},+\infty[\ \to X\times Y$ be a solution of \eqref{sec1:arrhurtik}, and suppose that $\varepsilon\notin\LL^{1}([t_{0},+\infty[)$ with either $\dot{\varepsilon}\in\LL^{1}([t_{0},+\infty[)$ or $\abs{\dot{\varepsilon}}^{2}/\varepsilon\in\LL^{1}([t_{0},+\infty[)$. Then it holds that
	\begin{equation*}
		\lim_{t\to+\infty}(x(t),\lambda(t))=\proj_{S\times M}(0,0).
	\end{equation*}
\end{proposition}
\begin{proof}
	In view of Corollary \ref{sec3:co:limsup} (see also Remark \ref{sec3:rm:liminf}), it suffices to show that
	\begin{equation*}
		\liminf_{t\to+\infty}\,\norm{(x(t),\lambda(t))}
		\geq\norm{\proj_{S\times M}(0,0)}.
	\end{equation*}
	Let $(x,\lambda)\in X\times Y$ and suppose that $(x(t_{n}),\lambda(t_{n}))\rightharpoonup(\bar{x},\bar{\lambda})$ weakly in $X\times Y$, as $n\to+\infty$, for a sequence $t_{n}\to+\infty$. By virtue of Theorem \ref{sec3:th:zerogap}, we have
	\begin{align*}
		0&=\lim_{n\to+\infty}
		\big(L(x(t_{n}),\lambda)-L(x,\lambda(t_{n}))\big)\\
		&\geq\liminf_{n\to+\infty}\,L(x(t_{n}),\lambda)
		+\liminf_{n\to+\infty}\,(-L(x,\lambda(t_{n})))\\
		&\geq L(\bar{x},\lambda)-L(x,\bar{\lambda})
	\end{align*}
	thanks to the weak lower semi-continuity of $L(\,\cdot\,,\lambda)$ and $-L(x,\,\cdot\,)$. The above inequalities being true for every $(x,\lambda)\in X\times Y$, we conclude that $(\bar{x},\bar{\lambda})\in S\times M$.

	On the other hand, the weak lower semi-continuity of the norm $\norm{\,\cdot\,}$ ensures
	\begin{equation*}
		\liminf_{n\to+\infty}\,\norm{(x(t_{n}),\lambda(t_{n}))}
		\geq\norm{(\bar{x},\bar{\lambda})}.
	\end{equation*}
	This inequality being true for every $(\bar{x},\bar{\lambda})\in S\times M$, taking the minimum over $S\times M$ yields the desired conclusion.
\end{proof}

Let us next provide a strong convergence result for the solutions of \eqref{sec1:arrhurtik} under the assumption that $\varepsilon(t)$ is twice continuously differentiable with
\begin{equation*}
	\lim_{t\to+\infty}\varepsilon(t)=0.
\end{equation*}
\begin{proposition}\label{sec3:pr:strongconv2}
	Let $S\times M$ be non-empty, let $(x,\lambda):[t_{0},+\infty[\ \to X\times Y$ be a solution of \eqref{sec1:arrhurtik}, and let $\varepsilon:[t_{0},+\infty[\ \to\ ]0,+\infty[$ be such that
	\begin{align*}
		\begin{rcases}
			\begin{aligned}
				\varepsilon^{2}(t)+\dot{\varepsilon}(t)&\geq0\\
				2\varepsilon(t)\dot{\varepsilon}(t)+\ddot{\varepsilon}(t)&\leq0
				\hspace{1pt}
			\end{aligned}
		\end{rcases}\quad\forall t\geq t_{+}
	\end{align*}
	for some $t_{+}\geq t_{0}$. Then it holds that
	\begin{equation*}
		\lim_{t\to+\infty}(x(t),\lambda(t))=\proj_{S\times M}(0,0).
	\end{equation*}
\end{proposition}
\begin{proof}
	By virtue of Cominetti et al.~\cite[Proposition 6]{RC-JP-SS:08}, it suffices to show that $\varepsilon\notin\LL^{1}([t_{+},+\infty[)$ and that all weak sequential cluster points of $(x(t),\lambda(t))_{t\geq t_{0}}$ belong to the set $S\times M$. Let $t_{+}\geq t_{0}$ be such that $1\geq-\dot{\varepsilon}(t)/\varepsilon^{2}(t)$ for all $t\geq t_{+}$. An immediate integration over $[t_{+},t]$ yields
	\begin{equation*}
		t-t_{+}\geq\frac{1}{\varepsilon(t)}-\frac{1}{\varepsilon(t_{+})}.
	\end{equation*}
	Integrating again and passing to the limit as $t\to+\infty$ gives
	\begin{equation*}
		\int_{t_{+}}^{\infty}\varepsilon(\tau)\d\tau\geq
		\int_{t_{+}}^{\infty}\frac{1}{\tau-t_{+}+\frac{1}{\varepsilon(t_{+})}}
		\d\tau=+\infty
	\end{equation*}
	so that $\varepsilon\notin\LL^{1}([t_{+},+\infty[)$.

	Let $(x,\lambda)\in X\times Y$ and suppose now that $(x(t_{n}),\lambda(t_{n}))\rightharpoonup(\bar{x},\bar{\lambda})$ weakly in $X\times Y$, as $n\to+\infty$, for a sequence $t_{n}\to+\infty$. Since $(x(t),\lambda(t))$ is bounded on $[t_{0},+\infty[$ and $(\dot{x}(t),\dot{\lambda}(t))\to(0,0)$ strongly in $X\times Y$ as $t\to+\infty$ (as it will be justified later in Proposition \ref{sec4:pr:velrate}), it follows from inequality \eqref{sec3:ttltineq} together with the regularized Lagrangian $L_{t}$ that
	\begin{align*}
		0&\geq\limsup_{n\to+\infty}
		\big(L(x_{t_{n}},\lambda)-L(x,\lambda_{t_{n}})\big)\\
		&\geq\liminf_{n\to+\infty}\,L(x_{t_{n}},\lambda)
		+\liminf_{n\to+\infty}\,(-L(x,\lambda_{t_{n}}))\\
		&\geq L(\bar{x},\lambda)-L(x,\bar{\lambda}),
	\end{align*}
	where we again utilized the weak lower semi-continuity of $L(\,\cdot\,,\lambda)$ and $-L(x,\,\cdot\,)$. The above derivations being true for every $(x,\lambda)\in X\times Y$, we conclude that $(\bar{x},\bar{\lambda})$ is a saddle point of $L$, that is, $(\bar{x},\bar{\lambda})\in S\times M$.
\end{proof}

\begin{remark}
	Under the hypotheses of Proposition \ref{sec3:pr:strongconv} or Proposition \ref{sec3:pr:strongconv2}, for every $(\xi,\eta)\in X\times Y$, we immediately observe that the solutions $(x,\lambda):[t_{0},+\infty[\ \to X\times Y$ of the nonautonomous differential system
	\begin{equation*}
		\begin{cases}
			\begin{aligned}
				\dot{x}+\nabla f(x)+A^{\ast}\lambda+\varepsilon(t)(x-\xi)&=0\\
				\dot{\lambda}+b-Ax+\varepsilon(t)(\lambda-\eta)&=0
			\end{aligned}
		\end{cases}
	\end{equation*}
	strongly converge toward $\proj_{S\times M}(\xi,\eta)$ as $t\to+\infty$. We leave the details to the reader.
\end{remark}

\section{Convergence rate estimates}\label{sec4}
In this section, we aim at deriving fast convergence rate estimates for the \eqref{sec1:arrhurtik} solutions. To this end, we again restrict the class of Tikhonov regularization functions by replacing assumption (A4) with the condition
\begin{enumerate}[\hspace{6pt}({A}4)$^{\prime}$]
	\item $\varepsilon:[t_{0},+\infty[\ \to \ ]0,+\infty[$ is twice
		continuously differentiable such that
	\begin{equation*}
		\lim_{t\to+\infty}\varepsilon(t)=0.
	\end{equation*}
\end{enumerate}

\subsection{Asymptotics relative to the set of zeros}
Let us begin our discussion by deriving fast decay rate estimates for the solutions of the \eqref{sec1:arrhurtik} differential system with respect to its set of zeros. The following result is based on the assumption that the Tikhonov regularization function $\varepsilon:[t_{0},+\infty[\ \to \ ]0,+\infty[$ verifies the decisive conditions
\begin{align*}
	\begin{rcases}
		\begin{aligned}
			\varepsilon^{2}(t)+\dot{\varepsilon}(t)&\geq0\\
			2\varepsilon(t)\dot{\varepsilon}(t)+\ddot{\varepsilon}(t)&\leq0
			\hspace{1pt}
		\end{aligned}
	\end{rcases}\quad\forall t\geq t_{+}
\end{align*}
for some $t_{+}\geq t_{0}$.
\begin{theorem}\label{sec4:th:convrate}
	Let $S\times M$ be non-empty, let $(x,\lambda):[t_{0},+\infty[\ \to X\times Y$ be a solution of \eqref{sec1:arrhurtik}, and let $\varepsilon:[t_{0},+\infty[\ \to\ ]0,+\infty[$ be such that
	\begin{align*}
		\begin{rcases}
			\begin{aligned}
				\varepsilon^{2}(t)+\dot{\varepsilon}(t)&\geq0\\
				2\varepsilon(t)\dot{\varepsilon}(t)+\ddot{\varepsilon}(t)&\leq0
				\hspace{1pt}
			\end{aligned}
		\end{rcases}\quad\forall t\geq t_{+}
	\end{align*}
	for some $t_{+}\geq t_{0}$. Then, for every $(\bar{x},\bar{\lambda})\in S\times M$, the following assertions hold:
	\begin{align*}
		\norm{(\dot{x}(t),\dot{\lambda}(t))+\varepsilon(t)((x(t),\lambda(t))
		-(\bar{x},\bar{\lambda}))}^{2}&=
		\OO\big(\e^{-2\rho(t)}+\,\varepsilon^{2}(t)\big)\
		\text{as $t\to+\infty$};\\
		\varepsilon(t)\big(L(x(t),\bar{\lambda})-L(\bar{x},\lambda(t))\big)&=
		\OO\big(\e^{-2\rho(t)}+\,\varepsilon^{2}(t)\big)\
		\text{as $t\to+\infty$};\\
		\norm{T(x(t),\lambda(t))-T(\bar{x},\bar{\lambda})}^{2}&=
		\OO\big(\e^{-2\rho(t)}+\,\varepsilon^{2}(t)\big)\
		\text{as $t\to+\infty$}.
	\end{align*}
\end{theorem}
\begin{proof}
	Let $(\bar{x},\bar{\lambda})\in S\times M$ and define $\psi:[t_{0},+\infty[\ \to\RR$ by $\psi(t)=\norm{(\dot{x}(t),\dot{\lambda}(t))+\varepsilon(t)((x(t),\lambda(t))-(\bar{x},\bar{\lambda}))}^{2}/2$. Moreover, let $\sigma:[t_{0},+\infty[\ \to\RR$ be defined by $\sigma(t)=\varepsilon^{2}(t)+\dot{\varepsilon}(t)$ such that $\dot{\sigma}(t)=2\varepsilon(t)\dot{\varepsilon}(t)+\ddot{\varepsilon}(t)$. Differentiating $\psi(t)$ and taking \eqref{sec1:arrhurtik} into account yields, for almost every $t\geq t_{0}$,
	\begin{equation*}
		\dot{\psi}(t)+\Big\langle\frac{\d}{\d t}T(x(t),\lambda(t))
		+\dot{\varepsilon}(t)(\bar{x},\bar{\lambda}),(\dot{x}(t),\dot{\lambda}(t))
		+\varepsilon(t)((x(t),\lambda(t))-(\bar{x},\bar{\lambda}))\Big\rangle=0.
	\end{equation*}
	In view of an immediate expansion, we obtain
	\begin{gather*}
		\Big\langle\frac{\d}{\d t}T(x(t),\lambda(t))
		+\dot{\varepsilon}(t)(\bar{x},\bar{\lambda}),(\dot{x}(t),\dot{\lambda}(t))
		+\varepsilon(t)((x(t),\lambda(t))-(\bar{x},\bar{\lambda}))\Big\rangle\\
		+\,\big\langle T(x(t),\lambda(t))+\varepsilon(t)(\bar{x},\bar{\lambda}),
		(\dot{x}(t),\dot{\lambda}(t))
		+\varepsilon(t)((x(t),\lambda(t))-(\bar{x},\bar{\lambda}))\big\rangle\\
		+\,\dot{\psi}(t)+2\varepsilon(t)\psi(t)=0.
	\end{gather*}
	Utilizing the basic identity
	\begin{gather*}
		\frac{\d}{\d t}\big(\varepsilon(t)
		\ip{T(x(t),\lambda(t))}{(x(t),\lambda(t))-(\bar{x},\bar{\lambda})}\big)
		=\varepsilon(t)\ip{T(x(t),\lambda(t))}{(\dot{x}(t),\dot{\lambda}(t))}\\
		\hspace{-5pt}+\,\varepsilon(t)\Big\langle\frac{\d}{\d t}T(x(t),\lambda(t)),
		(x(t),\lambda(t))-(\bar{x},\bar{\lambda})\Big\rangle
		+\dot{\varepsilon}(t)\ip{T(x(t),\lambda(t))}
		{(x(t),\lambda(t))-(\bar{x},\bar{\lambda})}
	\end{gather*}
	together with the fact that
	\begin{gather*}
		\frac{\d}{\d t}\big(\dot{\varepsilon}(t)
		\ip{(x(t),\lambda(t))-(\bar{x},\bar{\lambda})}{(\bar{x},\bar{\lambda})}\big)
		=\ddot{\varepsilon}(t)\ip{(x(t),\lambda(t))-(\bar{x},\bar{\lambda})}
		{(\bar{x},\bar{\lambda})}\\
		+\,\dot{\varepsilon}(t)\frac{\d}{\d t}
		\ip{(x(t),\lambda(t))-(\bar{x},\bar{\lambda})}{(\bar{x},\bar{\lambda})},
	\end{gather*}
	the above equality reads as
	\begin{gather*}
		\frac{\d}{\d t}\Big(\psi(t)+\varepsilon(t)\ip{T(x(t),\lambda(t))}
		{(x(t),\lambda(t))-(\bar{x},\bar{\lambda})}
		+\frac{\sigma(t)}{2}\norm{(x(t),\lambda(t))}^{2}
		-\frac{\varepsilon^{2}(t)}{2}\norm{(\bar{x},\bar{\lambda})}^{2}\Big)\\
		\hspace{-12pt}+\,2\varepsilon(t)\Big(\psi(t)+\varepsilon(t)
		\ip{T(x(t),\lambda(t))}{(x(t),\lambda(t))-(\bar{x},\bar{\lambda})}
		+\frac{\sigma(t)}{2}\norm{(x(t),\lambda(t))}^{2}
		-\frac{\varepsilon^{2}(t)}{2}\norm{(\bar{x},\bar{\lambda})}^{2}\Big)\\
		+\,\Big\langle\frac{\d}{\d t}T(x(t),\lambda(t)),
		(\dot{x}(t),\dot{\lambda}(t))\Big\rangle
		-\,\frac{\dot{\sigma}(t)}{2}\norm{(x(t),\lambda(t))}^{2}=0.
	\end{gather*}
	Multiplying by $\e^{2\rho(t)}$ and taking into account that $T$ is monotone entails
	\begin{gather*}
		\frac{\d}{\d t}\Big(\e^{2\rho(t)}
		\big(\psi(t)+\varepsilon(t)\ip{T(x(t),\lambda(t))}
		{(x(t),\lambda(t))-(\bar{x},\bar{\lambda})}\big)\Big)\\
		+\,\frac{\d}{\d t}\Big(\e^{2\rho(t)}
		\Big(\frac{\sigma(t)}{2}\norm{(x(t),\lambda(t))}^{2}
		-\frac{\varepsilon^{2}(t)}{2}
		\norm{(\bar{x},\bar{\lambda})}^{2}\Big)\Big)\\
		-\,\e^{2\rho(t)}\frac{\dot{\sigma}(t)}{2}\norm{(x(t),\lambda(t))}^{2}\leq0.
	\end{gather*}
	Integrating over $[t_{+},t]$ and subsequently dividing by $\e^{2\rho(t)}$, we find that there exists $K\geq0$ such that
	\begin{equation}
		\begin{gathered}\label{sec4:convrate}
			\psi(t)+\varepsilon(t)
			\ip{T(x(t),\lambda(t))}{(x(t),\lambda(t))-(\bar{x},\bar{\lambda})}
			+\frac{\sigma(t)}{2}\norm{(x(t),\lambda(t))}^{2}\\
			-\,\e^{-2\rho(t)}\int_{t_{+}}^{t}\e^{2\rho(\tau)}
			\frac{\dot{\sigma}(\tau)}{2}\norm{(x(\tau),\lambda(\tau))}^{2}\d\tau
			\leq K\e^{-2\rho(t)}+\,\frac{\varepsilon^{2}(t)}{2}
			\norm{(\bar{x},\bar{\lambda})}^{2}.
		\end{gathered}
	\end{equation}
	Noticing that we have both $\sigma(t)\geq0$ and $\dot{\sigma}(t)\leq0$ for all $t\geq t_{+}$, we deduce
	\begin{gather*}
		\psi(t)+\varepsilon(t)
		\ip{T(x(t),\lambda(t))}{(x(t),\lambda(t))-(\bar{x},\bar{\lambda})}\\
		\leq K\e^{-2\rho(t)}+\,\frac{\varepsilon^{2}(t)}{2}
		\norm{(\bar{x},\bar{\lambda})}^{2}.
	\end{gather*}
	Taking into account that $(\bar{x},\bar{\lambda})$ is a saddle point of $L$ and using the fact that
	\begin{equation*}
		\ip{T(x(t),\lambda(t))}{(x(t),\lambda(t))-(\bar{x},\bar{\lambda})}
		\geq L(x(t),\bar{\lambda})-L(\bar{x},\lambda(t)),
	\end{equation*}
	we obtain both
	\begin{gather*}
		\psi(t)\leq K\e^{-2\rho(t)}+\,\frac{\varepsilon^{2}(t)}{2}
		\norm{(\bar{x},\bar{\lambda})}^{2},\ \text{and}\\
		\varepsilon(t)\big(L(x(t),\bar{\lambda})-L(\bar{x},\lambda(t))\big)\leq
		K\e^{-2\rho(t)}+\,\frac{\varepsilon^{2}(t)}{2}
		\norm{(\bar{x},\bar{\lambda})}^{2}.
	\end{gather*}
	In view of \eqref{sec1:arrhurtik} and the fact that $T(\bar{x},\bar{\lambda})=(0,0)$, the remaining estimate now follows at once from the basic inequality
	\begin{align*}
		\norm{T(x(t),\lambda(t))-T(\bar{x},\bar{\lambda})}^{2}&=
		\norm{(\dot{x}(t),\dot{\lambda}(t))
			+\varepsilon(t)((x(t),\lambda(t))-(\bar{x},\bar{\lambda}))
			+\varepsilon(t)(\bar{x},\bar{\lambda})}^{2}\\
		&\leq\big(\sqrt{2\psi(t)}+\varepsilon(t)
		\norm{(\bar{x},\bar{\lambda})}\big)^{2},
	\end{align*}
	concluding the result.
\end{proof}

As an immediate consequence of the above result, we recover the decay rate estimate
\begin{equation*}
	\norm{(\dot{x}(t),\dot{\lambda}(t))}^{2}
	=\OO\big(\e^{-2\rho(t)}+\,\abs{\dot{\varepsilon}(t)}\big)\
	\text{as $t\to+\infty$}
\end{equation*}
recently obtained by Bo\c{t} and Nguyen \cite{RIB-DKN:24} in the context of general monotone operator flows with Tikhonov regularization.
\begin{proposition}\label{sec4:pr:velrate}
	Let $S\times M$ be non-empty, let $(x,\lambda):[t_{0},+\infty[\ \to X\times Y$ be a solution of \eqref{sec1:arrhurtik}, and let $\varepsilon:[t_{0},+\infty[\ \to\ ]0,+\infty[$ be such that
	\begin{align*}
		\begin{rcases}
			\begin{aligned}
				\varepsilon^{2}(t)+\dot{\varepsilon}(t)&\geq0\\
				2\varepsilon(t)\dot{\varepsilon}(t)+\ddot{\varepsilon}(t)&\leq0
				\hspace{1pt}
			\end{aligned}
		\end{rcases}\quad\forall t\geq t_{+}
	\end{align*}
	for some $t_{+}\geq t_{0}$. Then it holds that
	\begin{equation*}
		\norm{(\dot{x}(t),\dot{\lambda}(t))}^{2}=\OO\big(\e^{-2\rho(t)}
		+\,\varepsilon^{2}(t)\big)\ \text{as $t\to+\infty$}.
	\end{equation*}
\end{proposition}
\begin{proof}
	Let $(\bar{x},\bar{\lambda})\in S\times M$ and consider again $\vartheta:[t_{0},+\infty[\ \to\RR$ defined by $\vartheta(t)=\norm{(\dot{x}(t),\dot{\lambda}(t))}^{2}/2$. Recall from inequality \eqref{sec4:convrate} that for every $t\geq t_{+}$, we have
	\begin{gather*}
		\psi(t)+\varepsilon(t)
		\ip{T(x(t),\lambda(t))}{(x(t),\lambda(t))-(\bar{x},\bar{\lambda})}
		+\frac{\sigma(t)}{2}\norm{(x(t),\lambda(t))}^{2}\\
		-\,\e^{-2\rho(t)}\int_{t_{+}}^{t}\e^{2\rho(\tau)}
		\frac{\dot{\sigma}(\tau)}{2}\norm{(x(\tau),\lambda(\tau))}^{2}\d\tau
		\leq K\e^{-2\rho(t)}+\,\frac{\varepsilon^{2}(t)}{2}
		\norm{(\bar{x},\bar{\lambda})}^{2}.
	\end{gather*}
	Utilizing the fact that $\dot{\sigma}(t)\leq0$ for all $t\geq t_{+}$, we obtain
	\begin{gather*}
		\psi(t)+\varepsilon(t)
		\ip{T(x(t),\lambda(t))}{(x(t),\lambda(t))-(\bar{x},\bar{\lambda})}
		+\frac{\sigma(t)}{2}\norm{(x(t),\lambda(t))}^{2}\\
		\leq K\e^{-2\rho(t)}+\,\frac{\varepsilon^{2}(t)}{2}
		\norm{(\bar{x},\bar{\lambda})}^{2}.
	\end{gather*}
	Equivalently, in view of \eqref{sec1:arrhurtik}, the above inequality reads
	\begin{equation*}
		\vartheta(t)+\frac{\sigma(t)}{2}\norm{(x(t),\lambda(t))}^{2}
		\leq K\e^{-2\rho(t)}+\,\frac{\varepsilon^{2}(t)}{2}
		\norm{(x(t),\lambda(t))}^{2}.
	\end{equation*}
	Taking into account that $\sigma(t)\geq0$ for all $t\geq t_{+}$, we have
	\begin{equation*}
		\vartheta(t)\leq K\e^{-2\rho(t)}+\,\frac{\varepsilon^{2}(t)}{2}
		\norm{(x(t),\lambda(t))}^{2}.
	\end{equation*}
	Since $(x(t),\lambda(t))$ remains bounded on $[t_{0},+\infty[$, we conclude the result.
\end{proof}

The following result provides a more refined estimate for the velocity of an \eqref{sec1:arrhurtik} solution given the additional assumption that the Tikhonov regularization function $\varepsilon:[t_{0},+\infty[\ \to\ ]0,+\infty[$ is decreasing.
\begin{proposition}\label{sec4:pr:ewma}
	Let $S\times M$ be non-empty, let $(x,\lambda):[t_{0},+\infty[\ \to X\times Y$ be a solution of \eqref{sec1:arrhurtik}, and let $\varepsilon:[t_{0},+\infty[\ \to\ ]0,+\infty[$ be decreasing such that
	\begin{align*}
		\begin{rcases}
			\begin{aligned}
				\varepsilon^{2}(t)+\dot{\varepsilon}(t)&\geq0\\
				2\varepsilon(t)\dot{\varepsilon}(t)+\ddot{\varepsilon}(t)&\leq0
				\hspace{1pt}
			\end{aligned}
		\end{rcases}\quad\forall t\geq t_{+}
	\end{align*}
	for some $t_{+}\geq t_{0}$. Then the following assertion holds:
	\begin{equation*}
		\lim_{t\to+\infty}\e^{-\rho(t)}\int_{t_{+}}^{t}\e^{\rho(\tau)}
		\frac{1}{\varepsilon(\tau)}\norm{(\dot{x}(\tau),\dot{\lambda}(\tau))}^{2}
		\d\tau<+\infty.
	\end{equation*}
\end{proposition}
\begin{proof}
	Let $(\bar{x},\bar{\lambda})\in S\times M$ and consider again $\phi:[t_{0},+\infty[\ \to\RR$ defined by $\phi(t)=\norm{(x(t),\lambda(t))-(\bar{x},\bar{\lambda})}^{2}/2$. Observing that we have both $\sigma(t)\geq0$ and $\dot{\sigma}(t)\leq0$ for all $t\geq t_{+}$, it follows from inequality \eqref{sec4:convrate} and the monotonicity of $T$ that
	\begin{equation*}
		\psi(t)\leq K\e^{-2\rho(t)}+\,\frac{\varepsilon^{2}(t)}{2}
		\norm{(\bar{x},\bar{\lambda})}^{2}.
	\end{equation*}
	Equivalently, we have
	\begin{equation*}
		\varepsilon(t)\dot{\phi}(t)+\varepsilon^{2}(t)\phi(t)+\vartheta(t)
		\leq K\e^{-2\rho(t)}+\,\frac{\varepsilon^{2}(t)}{2}
		\norm{(\bar{x},\bar{\lambda})}^{2}.
	\end{equation*}
	Successively dividing by $\varepsilon(t)$ and multiplying by $\e^{\rho(t)}$ gives
	\begin{equation*}
		\frac{\d}{\d t}\big(\e^{\rho(t)}\phi(t)\big)
		+\e^{\rho(t)}\frac{1}{\varepsilon(t)}\vartheta(t)
		\leq\frac{K}{\varepsilon(t)}\e^{-\rho(t)}
		+\,\frac{1}{2}\frac{\d}{\d t}\e^{\rho(t)}\norm{(\bar{x},\bar{\lambda})}^{2}.
	\end{equation*}
	Integrating over $[t_{+},t]$ and subsequently dividing by $\e^{\rho(t)}$ entails
	\begin{equation}
		\begin{gathered}\label{sec4:expweight}
			\phi(t)+\e^{-\rho(t)}\int_{t_{+}}^{t}\e^{\rho(\tau)}
			\frac{1}{\varepsilon(\tau)}\vartheta(\tau)\d\tau
			\leq K\e^{-\rho(t)}\int_{t_{+}}^{t}\frac{1}{\varepsilon(\tau)}
			\e^{-\rho(\tau)}\d\tau\\
			+\,\e^{-(\rho(t)-\rho(t_{+}))}\phi(t_{+})
			+\frac{1}{2}\norm{(\bar{x},\bar{\lambda})}^{2}.
		\end{gathered}
	\end{equation}

	On the other hand, since $\dot{\sigma}(t)\leq0$ for all $t\geq t_{+}$, it follows from an immediate integration that
	\begin{equation*}
		-\,\dot{\varepsilon}(t)\e^{2\rho(t)}\geq-\,\dot{\varepsilon}(t_{+})
		\e^{2\rho(t_{+})}.
	\end{equation*}
	Owing to the fact that $\varepsilon(t)$ is decreasing, we obtain
	\begin{equation*}
		-\frac{1}{\dot{\varepsilon}(t_{+})\e^{2\rho(t_{+})}}
		\frac{-\dot{\varepsilon}(t)}{\varepsilon(t)}
		\geq\frac{1}{\varepsilon(t)}\e^{-2\rho(t)}.
	\end{equation*}
	Multiplying by $\e^{\rho(t)}$ and taking into account that $\sigma(t)\geq0$ for all $t\geq t_{+}$ yields
	\begin{equation*}
		-\frac{1}{\dot{\varepsilon}(t_{+})\e^{2\rho(t_{+})}}
		\frac{\d}{\d t}\e^{\rho(t)}\geq\frac{1}{\varepsilon(t)}\e^{-\rho(t)}.
	\end{equation*}
	Integrating over $[t_{+},t]$ and subsequently dividing by $\e^{\rho(t)}$ gives
	\begin{equation*}
		-\frac{1}{\dot{\varepsilon}(t_{+})\e^{2\rho(t_{+})}}\geq\e^{-\rho(t)}\int_{t_{+}}^{t}\frac{1}{\varepsilon(\tau)}\e^{-\rho(\tau)}\d\tau,
	\end{equation*}
	implying that the integral on the right-hand side of inequality \eqref{sec4:expweight} remains bounded on $[t_{+},+\infty[$. Taking into account that $\phi(t)\geq0$, passing to the limit in inequality \eqref{sec4:expweight} as $t\to+\infty$ then gives the desired result.
\end{proof}

Let us now investigate the rate of strong convergence of the \eqref{sec1:arrhurtik} solutions under the assumption that there exists $\alpha>0$ such that for every $(x,\lambda),(\xi,\eta)\in X\times Y$, it holds that
\begin{equation}
	\renewcommand{\theequation}{L}\tag{\theequation}
	\norm{T(x,\lambda)-T(\xi,\eta)}^{2}
	\geq\alpha\norm{(x,\lambda)-(\xi,\eta)}^{2}.
\end{equation}
Recall that condition \eqref{sec1:errorbound} clearly implies that the set $S\times M$ is a singleton. The following result is an immediate consequence of Theorem \ref{sec4:th:convrate}.
\begin{corollary}\label{sec4:co:fastbound}
	Under the hypotheses of Theorem \ref{sec4:th:convrate}, suppose that $T:X\times Y\to X\times Y$ satisfies condition \eqref{sec1:errorbound}. Then, for $(\bar{x},\bar{\lambda})\in S\times M$, it holds that
	\begin{equation*}
		\norm{(x(t),\lambda(t))-(\bar{x},\bar{\lambda})}^{2}=
		\OO\big(\e^{-2\rho(t)}+\,\varepsilon^{2}(t)\big)\
		\text{as $t\to+\infty$}.
	\end{equation*}
\end{corollary}
\begin{proof}
	Let $(\bar{x},\bar{\lambda})\in S\times M$ and consider again $\phi:[t_{0},+\infty[\ \to\RR$ defined by $\phi(t)=\norm{(x(t),\lambda(t))-(\bar{x},\bar{\lambda})}^{2}/2$. Since $T$ satisfies condition \eqref{sec1:errorbound}, there exists $\alpha>0$ such that for every $t\geq t_{+}$, it holds that
	\begin{equation*}
		\norm{T(x(t),\lambda(t))-T(\bar{x},\bar{\lambda})}^{2}\geq
		2\alpha\phi(t).
	\end{equation*}
	The desired estimate now follows at once from Theorem \ref{sec4:th:convrate}.
\end{proof}

\subsection{Asymptotics relative to the viscosity curve}
Let us now adapt the previous results to obtain fast decay rate estimates for the \eqref{sec1:arrhurtik} solutions with respect to the viscosity curve $(x_{t},\lambda_{t})$ as $t\to+\infty$. Recall that the viscosity curve $(x_{t},\lambda_{t})$ is characterized, for each $t\geq t_{0}$, as the unique zero of the $\varepsilon(t)$-strongly monotone operator
\begin{equation*}
	T_{t}=T+\varepsilon(t)\id.
\end{equation*}

The following result is analogous to Theorem \ref{sec4:th:convrate}.
\begin{theorem}\label{sec4:th:convratevc}
	Let $S\times M$ be non-empty and let $(x,\lambda):[t_{0},+\infty[\ \to X\times Y$ be a solution of \eqref{sec1:arrhurtik}. Let $(x_{t},\lambda_{t})=\zer T_{t}$ for each $t\geq t_{0}$ and suppose that $\varepsilon:[t_{0},+\infty[\ \to\ ]0,+\infty[$ verifies
	\begin{align*}
		\begin{rcases}
			\begin{aligned}
				\varepsilon^{2}(t)+\dot{\varepsilon}(t)&\geq0\\
				2\varepsilon(t)\dot{\varepsilon}(t)+\ddot{\varepsilon}(t)&\leq0
				\hspace{1pt}
			\end{aligned}
		\end{rcases}\quad\forall t\geq t_{+}
	\end{align*}
	for some $t_{+}\geq t_{0}$. Then the following assertions hold:
	\begin{align*}
		\norm{(\dot{x}(t),\dot{\lambda}(t))+\varepsilon(t)((x(t),\lambda(t))
			-(x_{t},\lambda_{t}))}^{2}&=
		\OO\big(\e^{-2\rho(t)}+\,\varepsilon^{2}(t)\big)\
		\text{as $t\to+\infty$};\\
		\norm{T(x(t),\lambda(t))-T(x_{t},\lambda_{t})}^{2}&=
		\OO\big(\e^{-2\rho(t)}+\,\varepsilon^{2}(t)\big)\
		\text{as $t\to+\infty$}.
	\end{align*}
\end{theorem}
\begin{proof}
	Let $S\times M$ be non-empty, let $(x_{t},\lambda_{t})=\zer T_{t}$, and let $\theta:[t_{0},+\infty[\ \to\RR$ be defined by $\theta(t)=\norm{(\dot{x}(t),\dot{\lambda}(t)+\varepsilon(t)((x(t),\lambda(t))-(x_{t},\lambda_{t}))}^{2}/2$. Consider again $\sigma:[t_{0},+\infty[\ \to\RR$ defined by $\sigma(t)=\varepsilon^{2}(t)+\dot{\varepsilon}(t)$ such that $\dot{\sigma}(t)=2\varepsilon(t)\dot{\varepsilon}(t)+\ddot{\varepsilon}(t)$. Using similar derivations as in the proof of Theorem \ref{sec4:th:convrate}, there exists $K\geq0$ such that for every $t\geq t_{+}$, it holds that
	\begin{gather*}
		\theta(t)+\varepsilon(t)
		\ip{T(x(t),\lambda(t))}{(x(t),\lambda(t))-(x_{t},\lambda_{t})}
		+\frac{\sigma(t)}{2}\norm{(x(t),\lambda(t))}^{2}\\
		-\,\e^{-2\rho(t)}\int_{t_{+}}^{t}\e^{2\rho(\tau)}
		\frac{\dot{\sigma}(\tau)}{2}\norm{(x(\tau),\lambda(\tau))}^{2}\d\tau
		\leq K\e^{-2\rho(t)}+\,\frac{\varepsilon^{2}(t)}{2}
		\norm{(x_{t},\lambda_{t})}^{2}.
	\end{gather*}
	Observing that we have both $\sigma(t)\geq0$ and $\dot{\sigma}(t)\leq0$ for all $t\geq t_{+}$, we infer
	\begin{gather*}
		\theta(t)+\varepsilon(t)
		\ip{T(x(t),\lambda(t))}{(x(t),\lambda(t))-(x_{t},\lambda_{t})}\\
		\leq K\e^{-2\rho(t)}+\,\frac{\varepsilon^{2}(t)}{2}
		\norm{(x_{t},\lambda_{t})}^{2}.
	\end{gather*}
	In view of the system of primal-dual optimality conditions, we readily deduce
	\begin{gather*}
		\theta(t)+\varepsilon(t)\ip{T(x(t),\lambda(t))-T(x_{t},\lambda_{t})}
		{(x(t),\lambda(t))-(x_{t},\lambda_{t})}\\
		+\,\varepsilon^{2}(t)\zeta(t)\leq K\e^{-2\rho(t)}
		+\,\frac{\varepsilon^{2}(t)}{2}\norm{(x(t),\lambda(t))}^{2}
	\end{gather*}
	with $\zeta:[t_{0},+\infty[\ \to\RR$ being defined as $\zeta(t)=\norm{(x(t),\lambda(t))-(x_{t},\lambda_{t})}^{2}/2$. Using the fact that $T$ is monotone and taking into account that $(x(t),\lambda(t))$ remains bounded on $[t_{0},+\infty[$, we arrive at the desired conclusion.
\end{proof}

Similarly to Proposition \ref{sec4:pr:ewma}, some more refined estimates can be derived under the additional assumption that the Tikhonov regularization function $\varepsilon:[t_{0},+\infty[\ \to\ ]0,+\infty[$ is decreasing.
\begin{proposition}
	Let $S\times M$ be non-empty and let $(x,\lambda):[t_{0},+\infty[\ \to X\times Y$ be a solution of \eqref{sec1:arrhurtik}. Let $(x_{t},\lambda_{t})=\zer T_{t}$ for each $t\geq t_{0}$ and suppose that $\varepsilon:[t_{0},+\infty[\ \to\ ]0,+\infty[$ is decreasing such that
	\begin{align*}
		\begin{rcases}
			\begin{aligned}
				\varepsilon^{2}(t)+\dot{\varepsilon}(t)&\geq0\\
				2\varepsilon(t)\dot{\varepsilon}(t)+\ddot{\varepsilon}(t)&\leq0
				\hspace{1pt}
			\end{aligned}
		\end{rcases}\quad\forall t\geq t_{+}
	\end{align*}
	for some $t_{+}\geq t_{0}$. Then the following estimates are verified:
	\begin{align*}
		\lim_{t\to+\infty}\e^{-\rho(t)}\int_{t_{+}}^{t}\e^{\rho(\tau)}
		\frac{1}{\varepsilon(\tau)}\norm{(\dot{x}(\tau),\dot{\lambda}(\tau))
			+\varepsilon(t)((x(\tau),\lambda(\tau))-(x_{\tau},\lambda_{\tau}))}^{2}
		\d\tau&<+\infty;\\
		\lim_{t\to+\infty}\e^{-\rho(t)}\int_{t_{+}}^{t}\e^{\rho(\tau)}
		\frac{1}{\varepsilon(\tau)}
		\norm{T(x(\tau),\lambda(\tau))-T(x_{\tau},\lambda_{\tau})}^{2}
		\d\tau&<+\infty.
	\end{align*}
\end{proposition}
\begin{proof}
	Let $S\times M$ be non-empty, let $(x_{t},\lambda_{t})=\zer T_{t}$, and let $\theta:[t_{0},+\infty[\ \to\RR$ be defined by $\theta(t)=\norm{(\dot{x}(t),\dot{\lambda}(t)+\varepsilon(t)((x(t),\lambda(t))-(x_{t},\lambda_{t}))}^{2}/2$. In view of \eqref{sec1:arrhurtik} and the system of primal-dual optimality conditions, for every $t\geq t_{0}$, we have
	\begin{align*}
		\norm{(\dot{x}(t),\dot{\lambda}(t))}^{2}&=
		\norm{T(x(t),\lambda(t))-T(x_{t},\lambda_{t})
			+\varepsilon(t)((x(t),\lambda(t))-(x_{t},\lambda_{t}))}^{2}\\
		&\geq\norm{T(x(t),\lambda(t))-T(x_{t},\lambda_{t})}^{2}
		+\varepsilon^{2}(t)\norm{(x(t),\lambda(t))-(x_{t},\lambda_{t})}^{2}\\
		&\geq\norm{T(x(t),\lambda(t))-T(x_{t},\lambda_{t})}^{2},
	\end{align*}
	where the first inequality follows from the monotonicity of $T$. Upon using \eqref{sec1:arrhurtik} again, we infer
	\begin{equation*}
		\norm{(\dot{x}(t),\dot{\lambda}(t))}^{2}\geq2\theta(t).
	\end{equation*}
	The assertions are now readily deduced as in Proposition \ref{sec4:pr:ewma}.
\end{proof}

Finally, let us provide an estimate on the \eqref{sec1:arrhurtik} solutions relative to the viscosity curve $(x_{t},\lambda_{t})$ assuming again that $T:X\times Y\to X\times Y$ verifies condition \eqref{sec1:errorbound}.
\begin{corollary}
	Under the hypotheses of Theorem \ref{sec4:th:convratevc}, suppose that $T:X\times Y\to X\times Y$ satisfies condition \eqref{sec1:errorbound}. Then the following assertion holds:
	\begin{equation*}
		\norm{(x(t),\lambda(t))-(x_{t},\lambda_{t})}^{2}=
		\OO\big(\e^{-2\rho(t)}+\,\varepsilon^{2}(t)\big)\
		\text{as $t\to+\infty$}.
	\end{equation*}
\end{corollary}
\begin{proof}
	In view of condition \eqref{sec1:errorbound}, the assertion follows at once from Theorem \ref{sec4:th:convratevc}.
\end{proof}

\subsection{The particular case $\varepsilon(t)=1/t^{p}$}
Let us now particularize the previous results to the case when the Tikhonov regularization function $\varepsilon:[t_{0},+\infty[\ \to\ ]0,+\infty[$ takes the form
\begin{equation*}
	\varepsilon(t)=\frac{1}{t^{p}}\quad\text{with $p\in\ ]0,1]$ and $t_{0}>0$}.
\end{equation*}
Since $\varepsilon(t)$ vanishes as $t\to+\infty$ with $\varepsilon\notin\LL^{1}([t_{0},+\infty[)$ and $\dot{\varepsilon}\in\LL^{1}([t_{0},+\infty[)$ for every $p\in\ ]0,1]$, we immediately deduce from Proposition \ref{sec3:pr:strongconv} that the solutions $(x(t),\lambda(t))$ of \eqref{sec1:arrhurtik} strongly converge toward $\proj_{S\times M}(0,0)$ as $t\to+\infty$. Moreover, for every $t\geq t_{0}$, we have $\dot{\varepsilon}(t)=-p/t^{p+1}$ and $\ddot{\varepsilon}(t)=p(p+1)/t^{p+2}$ so that
\begin{align*}
	\varepsilon^{2}(t)+\dot{\varepsilon}(t)
	&=\frac{1}{t^{2p}}-\frac{p}{t^{p+1}};\\
	2\varepsilon(t)\dot{\varepsilon}(t)+\ddot{\varepsilon}(t)
	&=-\frac{2p}{t^{2p+1}}+\frac{p(p+1)}{t^{p+2}}.
\end{align*}
In the case $p=1$, we have both $\varepsilon^{2}(t)+\dot{\varepsilon}(t)=0$ and $2\varepsilon(t)\dot{\varepsilon}(t)+\ddot{\varepsilon}(t)=0$ for every $t\geq t_{0}$. On the other hand, whenever $p\in\ ]0,1[$, we readily obtain
\begin{align*}
	\frac{1}{t^{2p}}-\frac{p}{t^{p+1}}&\geq0\quad\iff\quad
	\hspace{-7pt}\sqrt[\leftroot{1}\uproot{3}1-p]{p}\leq t;\\
	-\frac{2p}{t^{2p+1}}+\frac{p(p+1)}{t^{p+2}}&\leq0\quad\iff\quad
	\hspace{-7pt}\sqrt[\leftroot{1}\uproot{3}1-p]{\frac{p+1}{2}}\leq t.
\end{align*}
Consequently, for every $p\in\ ]0,1]$, there exists $t_{+}=\max\,\{t_{0},\hspace{-7pt}\sqrt[\leftroot{1}\uproot{3}1-p]{(p+1)/2}\}$ such that
\begin{align*}
	\begin{rcases}
		\begin{aligned}
			\varepsilon^{2}(t)+\dot{\varepsilon}(t)&\geq0\\
			2\varepsilon(t)\dot{\varepsilon}(t)+\ddot{\varepsilon}(t)&\leq0
			\hspace{1pt}
		\end{aligned}
	\end{rcases}\quad\forall t\geq t_{+},
\end{align*}
implying that the hypotheses of Theorem \ref{sec4:th:convrate} are verified. This immediately leads to the following assertion.
\begin{proposition}
	Let $S\times M$ be non-empty, let $(x,\lambda):[t_{0},+\infty[\ \to X\times Y$ be a solution of \eqref{sec1:arrhurtik}, and let $\varepsilon:[t_{0},+\infty[\ \to\ ]0,+\infty[$ be defined by $\varepsilon(t)=1/t^{p}$ with $p\in\ ]0,1]$ and $t_{0}>0$. Then, for every $(\bar{x},\bar{\lambda})\in S\times M$, it holds that
	\begin{align*}
		\norm{(\dot{x}(t),\dot{\lambda}(t))+\frac{1}{t^{p}}((x(t),\lambda(t))
		-(\bar{x},\bar{\lambda}))}^{2}&=
		\OO\Big(\e^{-2\rho(t)}+\,\frac{1}{t^{2p}}\Big)\
		\text{as $t\to+\infty$};\\
		\frac{1}{t^{p}}\big(L(x(t),\bar{\lambda})-L(\bar{x},\lambda(t))\big)&=
		\OO\Big(\e^{-2\rho(t)}+\,\frac{1}{t^{2p}}\Big)\
		\text{as $t\to+\infty$};\\
		\norm{T(x(t),\lambda(t))-T(\bar{x},\bar{\lambda})}^{2}&=
		\OO\Big(\e^{-2\rho(t)}+\,\frac{1}{t^{2p}}\Big)\
		\text{as $t\to+\infty$};\\
		\norm{(\dot{x}(t),\dot{\lambda}(t))}^{2}&=\OO\Big(\e^{-2\rho(t)}+\,\frac{1}{t^{2p}}\Big)\
		\text{as $t\to+\infty$}.
	\end{align*}
	Moreover, $(x(t),\lambda(t))$ converges strongly to $\proj_{S\times M}(0,0)$  as $t\to+\infty$.
\end{proposition}

\begin{remark}
	In view of the above result, we observe that the fastest rate of convergence is achieved for the value $p=1$. In this case, the above asymptotic estimates reduce to
	\begin{align*}
		\norm{(\dot{x}(t),\dot{\lambda}(t))+\frac{1}{t}((x(t),\lambda(t))
		-(\bar{x},\bar{\lambda}))}^{2}&=\OO\Big(\frac{1}{t^{2}}\Big)\
		\text{as $t\to+\infty$};\\
		\frac{1}{t}\big(L(x(t),\bar{\lambda})-L(\bar{x},\lambda(t))\big)&=
		\OO\Big(\frac{1}{t^{2}}\Big)\ \text{as $t\to+\infty$};\\
		\norm{T(x(t),\lambda(t))-T(\bar{x},\bar{\lambda})}^{2}&=
		\OO\Big(\frac{1}{t^{2}}\Big)\ \text{as $t\to+\infty$};\\
		\norm{(\dot{x}(t),\dot{\lambda}(t))}^{2}&=\OO\Big(\frac{1}{t^{2}}\Big)\
		\text{as $t\to+\infty$}.
	\end{align*}
	Moreover, if $T$ satisfies condition \eqref{sec1:errorbound}, it further holds that
	\begin{equation*}
		\norm{(x(t),\lambda(t))-(\bar{x},\bar{\lambda})}^{2}=
		\OO\Big(\frac{1}{t^{2}}\Big)\ \text{as $t\to+\infty$}.
	\end{equation*}
\end{remark}

With respect to the viscosity curve $(x_{t},\lambda_{t})$, we have the following decay rate estimates as $t\to+\infty$:
\begin{proposition}
	Let $S\times M$ be non-empty and let $(x,\lambda):[t_{0},+\infty[\ \to X\times Y$ be a solution of \eqref{sec1:arrhurtik}. Let $(x_{t},\lambda_{t})=\zer T_{t}$ for each $t\geq t_{0}$ and let $\varepsilon:[t_{0},+\infty[\ \to\ ]0,+\infty[$ be defined by $\varepsilon(t)=1/t^{p}$ with $p\in\ ]0,1]$ and $t_{0}>0$. Then the following assertions hold:
	\begin{align*}
		\norm{(\dot{x}(t),\dot{\lambda}(t))+\frac{1}{t^{p}}((x(t),\lambda(t))
			-(x_{t},\lambda_{t}))}^{2}&=
		\OO\Big(\e^{-2\rho(t)}+\,\frac{1}{t^{2p}}\Big)\
		\text{as $t\to+\infty$};\\
		\norm{T(x(t),\lambda(t))-T(x_{t},\lambda_{t})}^{2}&=
		\OO\Big(\e^{-2\rho(t)}+\,\frac{1}{t^{2p}}\Big)\
		\text{as $t\to+\infty$}.
	\end{align*}
\end{proposition}

\begin{remark}
	In the particular case $\varepsilon(t)=1/t$, and under the assumption that $T$ satisfies condition \eqref{sec1:errorbound}, we further have
	\begin{equation*}
		\norm{(x(t),\lambda(t))-(x_{t},\lambda_{t})}^{2}=
		\OO\Big(\frac{1}{t^{2}}\Big)\ \text{as $t\to+\infty$}.
	\end{equation*}
\end{remark}

\section{Numerical experiments}\label{sec5}
In this section, we provide a simple yet representative example that allows for a direct exposition of our main results.
\begin{example}
	Let $X,Y=\RR$ and consider the saddle-value problem
	\begin{equation*}
		\min_{x\in\RR}\max_{\lambda\in\RR}L(x,\lambda),
	\end{equation*}
	where $L:\RR\times\RR\to\RR$ is defined in terms of the convex-concave and continuously differentiable bifunction $L(x,\lambda)=\lambda(x-1)$. Let us choose the Tikhonov regularization function $\varepsilon:[t_{0},+\infty[\ \to\ ]0,+\infty[$ as $\varepsilon(t)=1/t^{p}$ with $p\in\ ]0,1]$ and $t_{0}>0$. In this case, the \eqref{sec1:arrhurtik} differential system reduces to
	\begin{equation*}
		\begin{cases}
			\begin{aligned}
				\dot{x}+\lambda+\frac{x}{t^{p}}&=0\\
				\dot{\lambda}+1-x+\frac{\lambda}{t^{p}}&=0.
			\end{aligned}
		\end{cases}
	\end{equation*}
	The evolution of the solutions $(x(t),\lambda(t))$ of the \eqref{sec1:arrhurtik} differential system together with the viscosity curve $(x_{t},\lambda_{t})$ as $t\to+\infty$ for different values of the Tikhonov regularization parameter $p\in\ ]0,1]$ is depicted in Figures \ref{sec5:ex:fig1} and \ref{sec5:ex:fig2}. We thereby distinguish the cases $p=1$ (see Figure \ref{sec5:ex:fig1}) and $p\in\ ]0,1[$ (see Figure \ref{sec5:ex:fig2}). In any case, the initial data is set to $(x_{0},\lambda_{0})=(0,0)$ with $t_{0}=1/100$.

	{\bf The particular case $p=1$.}
	Since $\varepsilon(t)$ tends to zero as $t\to+\infty$ with $\varepsilon\notin\LL^{1}([t_{0},+\infty[)$ and $\dot{\varepsilon}\in\LL^{1}([t_{0},+\infty[)$, we readily observe from Figure \ref{sec5:ex:fig1} that the solution $(x(t),\lambda(t))$ of \eqref{sec1:arrhurtik} converges, as $t\to+\infty$, to the unique saddle point $(\bar{x},\bar{\lambda})=(1,0)$ of the bifunction $L$; cf. Proposition \ref{sec3:pr:strongconv}. Moreover, since $\varepsilon(t)$ verifies, for every $t\geq t_{0}$, the decisive conditions
	\begin{align*}
		\begin{cases}
			\begin{aligned}
				\varepsilon^{2}(t)+\dot{\varepsilon}(t)&=0\\
				2\varepsilon(t)\dot{\varepsilon}(t)+\ddot{\varepsilon}(t)&=0,
			\end{aligned}
		\end{cases}
	\end{align*}
	we obtain, in accordance with Theorem \ref{sec4:th:convrate} and Proposition \ref{sec4:pr:velrate}, that $\norm{(\dot{x}(t),\dot{\lambda}(t))+\varepsilon(t)((x(t),\lambda(t))-(\bar{x},\bar{\lambda}))}^{2}$ and $\norm{(\dot{x}(t),\dot{\lambda}(t))}^{2}$ obey the asymptotic estimate $\OO(1/t^{2})$ as $t\to+\infty$. In particular, as $\varepsilon(t)$ is decreasing, we have the refined estimate
	\begin{equation*}
		\lim_{t\to+\infty}\frac{1}{t}\int_{t_{0}}^{t}\tau^{2}
		\norm{(\dot{x}(\tau),\dot{\lambda}(\tau))}^{2}\d\tau<+\infty,
	\end{equation*}
	which suggests that $t\norm{(\dot{x}(t),\dot{\lambda}(t))}^{2}$ vanishes fast as $t\to+\infty$ in the sense of an ``exponentially weighted moving average''; cf. Proposition \ref{sec4:pr:ewma}. A similar behavior can be observed for $\norm{(x(t),\lambda(t))-(x_{t},\lambda_{t})}^{2}$. Finally, since the operator
	\begin{align*}
		T:\RR\times\RR&\longrightarrow\RR\times\RR\\
		(x,\lambda)&\longmapsto(\lambda,1-x)
	\end{align*}
	associated with the saddle function $L$ verifies condition \eqref{sec1:errorbound} with $\alpha=1$, we readily observe that $\norm{(x(t),\lambda(t))-(\bar{x},\bar{\lambda})}^{2}$ obeys the asymptotic estimate $\OO(1/t^{2})$ as $t\to+\infty$; see Corollary \ref{sec4:co:fastbound}. In this scenario, we also find that $\norm{(\dot{x}_{t},\dot{\lambda}_{t})}^{2}$ vanishes, as predicted by Lemma \ref{sec2:lm:vccondl}, according to the fast asymptotic estimate $\OO(1/(t^{2}+1)^2)$ as $t\to+\infty$.

	{\bf The particular case $p\in\ ]0,1[$.} Analyzing Figure \ref{sec5:ex:fig2}, we observe that the solutions $(x(t),\lambda(t))$ of the \eqref{sec1:arrhurtik} differential system still admit favorable convergence properties, but their decay rate is considerably degraded as the value of the Tikhonov regularization parameter $p$ decreases. As predicted by Theorem \ref{sec4:th:convrate} and Proposition \ref{sec4:pr:velrate}, we find that $\norm{(\dot{x}(t),\dot{\lambda}(t))+\varepsilon(t)((x(t),\lambda(t))-(\bar{x},\bar{\lambda}))}^{2}$ and $\norm{(\dot{x}(t),\dot{\lambda}(t))}^{2}$ obey the asymptotic estimate $\OO\big(\e^{-2\rho(t)}+\,1/t^{2p}\big)$ as $t\to+\infty$. However, this estimate is no longer sharp for $p\in\ ]0,1[$ due to the conservatism introduced by the inequalities
	\begin{align*}
		\begin{rcases}
			\begin{aligned}
				\varepsilon^{2}(t)+\dot{\varepsilon}(t)&\geq0\\
				2\varepsilon(t)\dot{\varepsilon}(t)+\ddot{\varepsilon}(t)&\leq0
				\hspace{1pt}
			\end{aligned}
		\end{rcases}\quad\forall t\geq t_{+}
	\end{align*}
	for some $t_{+}\geq t_{0}$. Given this observation, we recover the fact that the fastest convergence rate estimates are obtained whenever the Tikhonov regularization function $\varepsilon:[t_{0},+\infty[\ \to\ ]0,+\infty[$ is chosen according to the differential equation
	\begin{equation*}
		\dot{\varepsilon}(t)+\varepsilon^{2}(t)=0,
	\end{equation*}
	whose solutions take the form
	\begin{equation*}
		\varepsilon(t)=\frac{1}{t+c},\quad c>-t_{0}.
	\end{equation*}
	We leave the discussion on sharp asymptotic decay rate estimates for the solutions $(x(t),\lambda(t))$ of \eqref{sec1:arrhurtik} in the case when the Tikhonov regularization parameter $p$ is chosen in $]0,1[$ open for future investigations.
\end{example}

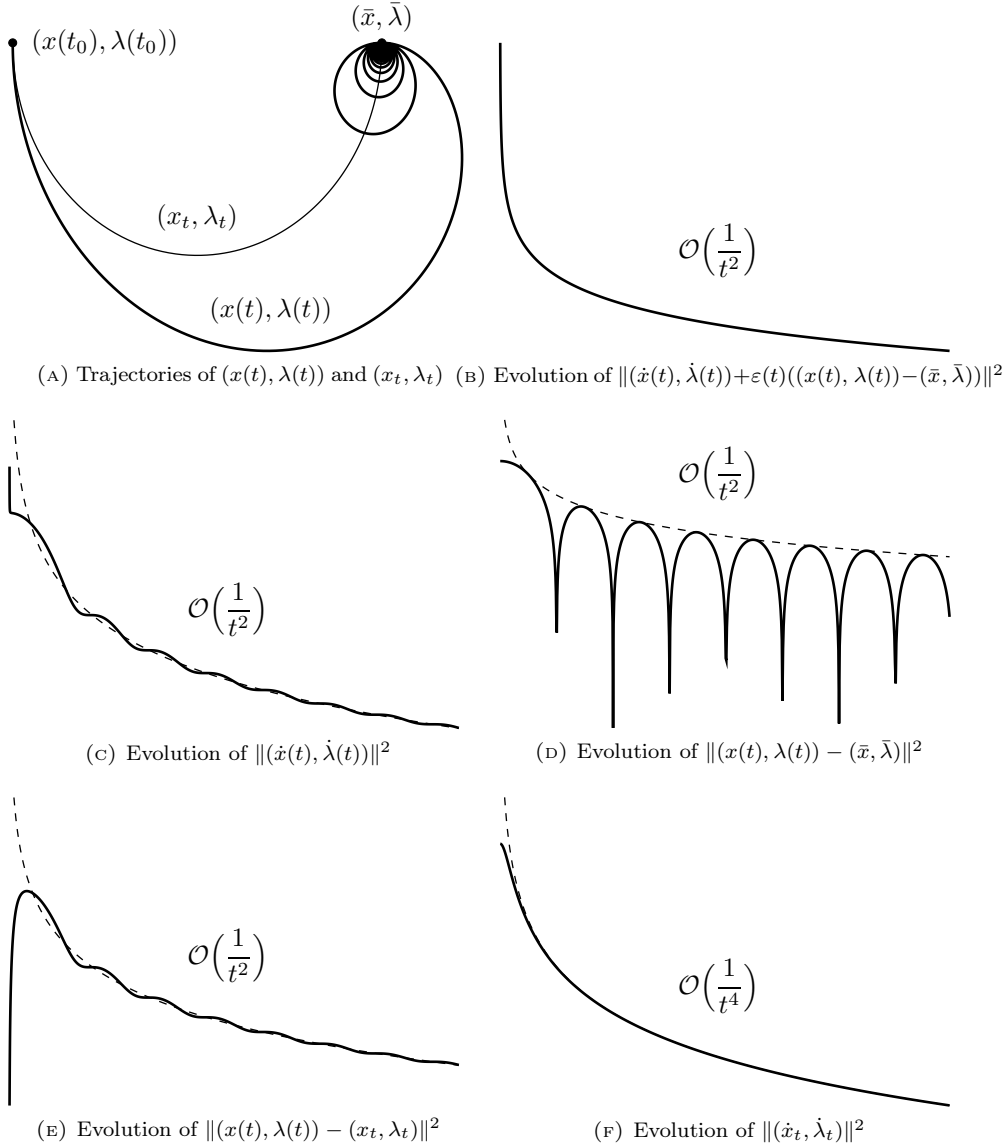
\begin{figure*}[!ht]
	\centering
	\captionsetup[subfigure]{width=154pt}%
	\subfloat[Trajectories of $(x(t),\lambda(t))$ and $(x_{t},\lambda_{t})$]{%
		\begin{tikzpicture}
			\pgfplotsset{width=8.75cm,height=6.5cm}
				\begin{axis}[hide axis,samples=250,clip=false]
					\addplot[line width=1pt] file {matlab/fig1a1.tex};
					\addplot[line width=.5pt] file {matlab/fig1a2.tex};
					\draw[fill=black] (0,0) circle (1.5pt);
					\draw[fill=black] (1,0) circle (1.5pt);
					\node at (.25,0) {$(x(t_{0}),\lambda(t_{0}))$};
					\node at (1,.06) {$(\bar{x},\bar{\lambda})$};
					\node at (.7,-.63) {$(x(t),\lambda(t))$};
					\node at (.5,-.41) {$(x_{t},\lambda_{t})$};
				\end{axis}
		\end{tikzpicture}
	}\hfill
	\captionsetup[subfigure]{width=208pt}%
	\subfloat[Evolution of $\norm{(\dot{x}(t),\dot{\lambda}(t))+\varepsilon(t)((x(t),\lambda(t))-(\bar{x},\bar{\lambda}))}^{2}$]{%
		\begin{tikzpicture}
			\pgfplotsset{width=8.75cm,height=6.5cm}
			\begin{axis}[hide axis,samples=250,ymode=log,clip=false]
				\addplot[line width=1pt] file {matlab/fig1b.tex};
			\end{axis}
			\node at (3.5,2.5) {$\OO\Big(\dfrac{1}{t^{2}}\Big)$};
		\end{tikzpicture}
	}\hfill
	\subfloat[Evolution of $\norm{(\dot{x}(t),\dot{\lambda}(t))}^{2}$]{%
		\begin{tikzpicture}
			\pgfplotsset{width=8.75cm,height=6.5cm}
			\begin{axis}[hide axis,samples=250,ymode=log,clip=false]
				\addplot[line width=1pt] file {matlab/fig1c.tex};
				\addplot[domain=.5:50,smooth,dashed,line width=.5pt]{.5/x^2};
			\end{axis}
			\node at (3.5,2.5) {$\OO\Big(\dfrac{1}{t^{2}}\Big)$};
		\end{tikzpicture}
	}\hfill
	\subfloat[Evolution of $\norm{(x(t),\lambda(t))-(\bar{x},\bar{\lambda})}^{2}$]{%
		\begin{tikzpicture}
			\pgfplotsset{width=8.75cm,height=6.5cm}
			\begin{axis}[hide axis,samples=250,ymode=log,clip=false]
				\addplot[line width=1pt] file {matlab/fig1d.tex};
				\addplot[domain=.5:50,smooth,dashed,line width=.5pt]{2/x^2};
			\end{axis}
			\node at (3.5,3.8) {$\OO\Big(\dfrac{1}{t^{2}}\Big)$};
		\end{tikzpicture}
	}\hfill
	\subfloat[Evolution of $\norm{(x(t),\lambda(t))-(x_{t},\lambda_{t})}^{2}$]{%
		\begin{tikzpicture}
			\pgfplotsset{width=8.75cm,height=6.5cm}
			\begin{axis}[hide axis,samples=250,ymode=log,clip=false]
				\addplot[line width=1pt] file {matlab/fig1e.tex};
				\addplot[domain=1.2:50,smooth,dashed,line width=.5pt]{.5/x^2};
			\end{axis}
			\node at (3.5,2.8) {$\OO\Big(\dfrac{1}{t^{2}}\Big)$};
		\end{tikzpicture}
	}\hfill
	\subfloat[Evolution of $\norm{(\dot{x}_{t},\dot{\lambda}_{t})}^{2}$]{%
		\begin{tikzpicture}
			\pgfplotsset{width=8.75cm,height=6.5cm}
			\begin{axis}[hide axis,samples=250,ymode=log,clip=false]
				\addplot[line width=1pt] file {matlab/fig1f.tex};
			\end{axis}
			\node at (3.5,2.8) {$\OO\Big(\dfrac{1}{(t^{2}+1)^{2}}\Big)$};
		\end{tikzpicture}
	}
	\caption{%
		Graphical view on the evolution of a solution $(x(t),\lambda(t))$ of the \eqref{sec1:arrhurtik} differential system together with the viscosity curve $(x_{t},\lambda_{t})$ as $t\to+\infty$ for the Tikhonov regularization parameter $p=1$.
	}\label{sec5:ex:fig1}
	\vspace{-10pt}
\end{figure*}

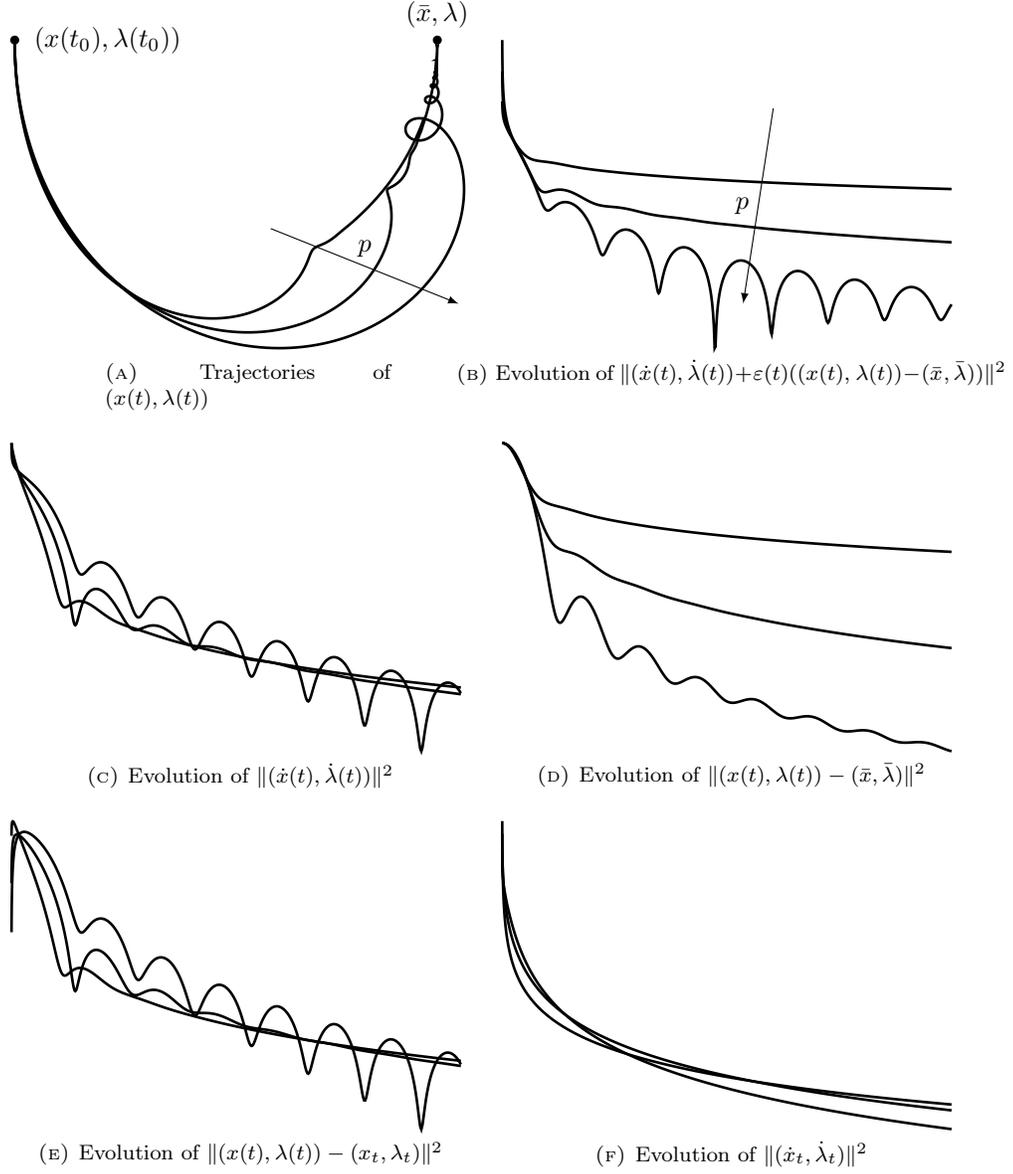
\begin{figure*}[!ht]
	\centering
	\captionsetup[subfigure]{width=154pt}%
	\subfloat[Trajectories of $(x(t),\lambda(t))$ and $(x_{t},\lambda_{t})$]{%
		\begin{tikzpicture}
			\pgfplotsset{width=8.75cm,height=6.5cm}
				\begin{axis}[hide axis,samples=250,clip=false]
					\addplot[line width=.5pt] file {matlab/fig1a2.tex};
					\addplot[line width=1pt] file {matlab/fig2a1.tex};
					\addplot[line width=1pt] file {matlab/fig2a2.tex};
					\addplot[line width=1pt] file {matlab/fig2a3.tex};
					\draw[fill=black] (0,0) circle (1.5pt);
					\draw[fill=black] (1,0) circle (1.5pt);
					\node at (.22,0) {$(x(t_{0}),\lambda(t_{0}))$};
					\node at (1,.06) {$(\bar{x},\bar{\lambda})$};
					\node at (.5,-.41) {$(x_{t},\lambda_{t})$};
				\end{axis}
				\draw[-latex] (4.3,2.3) -- node[above] {$p$} (6.1,.7);
		\end{tikzpicture}
	}\hfill
	\captionsetup[subfigure]{width=208pt}%
	\subfloat[Evolution of $\norm{(\dot{x}(t),\dot{\lambda}(t))+\varepsilon(t)((x(t),\lambda(t))-(\bar{x},\bar{\lambda}))}^{2}$]{%
	\begin{tikzpicture}
		\pgfplotsset{width=8.75cm,height=6.5cm}
		\begin{axis}[hide axis,samples=250,ymode=log,clip=false]
			\addplot[line width=1pt] file {matlab/fig2b1.tex};
			\addplot[line width=1pt] file {matlab/fig2b2.tex};
			\addplot[line width=1pt] file {matlab/fig2b3.tex};
		\end{axis}
		\draw[-latex] (4.2,3.6) -- node[left] {$p$} (3.8,1);
	\end{tikzpicture}
	}\hfill
	\subfloat[Evolution of $\norm{(\dot{x}(t),\dot{\lambda}(t))}^{2}$]{%
		\begin{tikzpicture}
			\pgfplotsset{width=8.75cm,height=6.5cm}
			\begin{axis}[hide axis,samples=250,ymode=log,clip=false]
				\addplot[line width=1pt] file {matlab/fig2c1.tex};
				\addplot[line width=1pt] file {matlab/fig2c2.tex};
				\addplot[line width=1pt] file {matlab/fig2c3.tex};
			\end{axis}
		\end{tikzpicture}
	}\hfill
	\subfloat[Evolution of $\norm{(x(t),\lambda(t))-(\bar{x},\bar{\lambda})}^{2}$]{%
		\begin{tikzpicture}
			\pgfplotsset{width=8.75cm,height=6.5cm}
			\begin{axis}[hide axis,samples=250,ymode=log,clip=false]
				\addplot[line width=1pt] file {matlab/fig2d1.tex};
				\addplot[line width=1pt] file {matlab/fig2d2.tex};
				\addplot[line width=1pt] file {matlab/fig2d3.tex};
			\end{axis}
			\draw[-latex] (4.2,3.4) -- node[left] {$p$} (3.8,.6);
		\end{tikzpicture}
	}\hfill
	\subfloat[Evolution of $\norm{(x(t),\lambda(t))-(x_{t},\lambda_{t})}^{2}$]{%
		\begin{tikzpicture}
			\pgfplotsset{width=8.75cm,height=6.5cm}
			\begin{axis}[hide axis,samples=250,ymode=log,clip=false]
				\addplot[line width=1pt] file {matlab/fig2e1.tex};
				\addplot[line width=1pt] file {matlab/fig2e2.tex};
				\addplot[line width=1pt] file {matlab/fig2e3.tex};
			\end{axis}
		\end{tikzpicture}
	}\hfill
	\subfloat[Evolution of $\norm{(\dot{x}_{t},\dot{\lambda}_{t})}^{2}$]{%
		\begin{tikzpicture}
			\pgfplotsset{width=8.75cm,height=6.5cm}
			\begin{axis}[hide axis,samples=250,ymode=log,clip=false]
				\addplot[line width=1pt] file {matlab/fig2f1.tex};
				\addplot[line width=1pt] file {matlab/fig2f2.tex};
				\addplot[line width=1pt] file {matlab/fig2f3.tex};
			\end{axis}
		\end{tikzpicture}
	}
	\caption{%
			Graphical view on the evolution of the \eqref{sec1:arrhurtik} solutions $(x(t),\lambda(t))$ and the viscosity curve $(x_{t},\lambda_{t})$ for the Tikhonov regularization parameters $p=1/4$, $p=1/2$, and $p=3/4$.
		}\label{sec5:ex:fig2}
		\vspace{-10pt}
\end{figure*}

\section*{Appendix}\appendix\label{app}
\renewcommand\thetheorem{A.\arabic{theorem}}\setcounter{theorem}{0}
We collect here some auxiliary results which are used in the asymptotic analysis of the solutions of the \eqref{sec1:arrhurtik} differential system.

Let us first recall the following classical result as outlined in Cominetti et al.~\cite[Lemma 1]{RC-JP-SS:08}.
\begin{lemma}\label{app:lm:lsineq}
	Let $\phi:[t_{0},+\infty[\ \to\RR$ be continuously differentiable, let $\vartheta:[t_{0},+\infty[\ \to\RR$ be bounded, and let $\varepsilon:[t_{0},+\infty[\ \to[0,+\infty[$ be locally integrable such that
	\begin{equation*}
		\dot{\phi}(t)+\varepsilon(t)\phi(t)\leq\varepsilon(t)\vartheta(t)\quad
		\forall t\geq t_{0}.
	\end{equation*}
	Then $\phi(t)$ remains bounded on $[t_{0},+\infty[$. Moreover, if $\varepsilon\notin\LL^{1}([t_{0},+\infty[)$, then
	\begin{equation*}
		\limsup_{t\to+\infty}\,\phi(t)\leq\limsup_{t\to+\infty}\,\vartheta(t).
	\end{equation*}
\end{lemma}

For the following basic inequality of Gronwall-type, the reader is referred to Br{\'e}zis \cite[Lemma A.5]{HB:73}.
\begin{lemma}\label{app:lm:gtineq}
	Let $\phi:[t_{0},+\infty[\ \to\RR$ be continuous and non-negative, and let $\vartheta:[t_{0},+\infty[\ \to[0,+\infty[$ be locally integrable such that
	\begin{equation*}
		\frac{1}{2}\phi^{2}(t)\leq\frac{1}{2}\phi^{2}(t_{0})
		+\int_{t_{0}}^{t}\vartheta(\tau)\phi(\tau)\d\tau\quad\forall t\geq t_{0}.
	\end{equation*}
	Then it holds that
	\begin{equation*}
		\phi(t)\leq\phi(t_{0})+\int_{t_{0}}^{t}\vartheta(\tau)\d\tau\quad
		\forall t\geq t_{0}.
	\end{equation*}
\end{lemma}

\bibliographystyle{amsplain} 
\bibliography{alias,main}

\end{document}